\documentclass[11pt,dvips,twoside,letterpaper]{article}
\usepackage{pslatex}
\usepackage{fancyhdr}
\usepackage{graphicx}
\usepackage{geometry}

\def\figurename{Figure} 
\makeatletter
\renewcommand{\fnum@figure}[1]{\figurename~\thefigure.}
\makeatother

\def\tablename{Table} 
\makeatletter
\renewcommand{\fnum@table}[1]{\tablename~\thetable.}
\makeatother

\usepackage{amsmath}
\usepackage{amssymb}
\usepackage{amsfonts}
\usepackage{amsthm,amscd}

\newtheorem{theorem}{Theorem}[section]
\newtheorem{lemma}{Lemma}[section]

\newtheorem{proposition}{Proposition}[section]
\theoremstyle{definition}
\newtheorem{definition}{Definition}[section]

\theoremstyle{remark}
\newtheorem{remark}{Remark}[section]

\numberwithin{equation}{section}

\def\P{\mathbb P}
\def\R{\mathbb R}
\def\E{\mathbb E}

\def\E{\mathbb E}
\def\N{\mathbb N}
\def\cal{\mathcal}

\begin{document}
\title{\bfseries\scshape{Obstacle problem for SPDE with nonlinear Neumann boundary condition via reflected generalized backward doubly SDEs}}
\author{\bfseries\scshape Auguste Aman\thanks{Supported by AUF post doctoral grant 07-08, Réf:PC-420/2460}\;\;\thanks{augusteaman5@yahoo.fr,\ Corresponding author.}\\
U.F.R.M.I, Universit\'{e} de Cocody, \\582 Abidjan 22, C\^{o}te d'Ivoire\\
\\\bfseries\scshape N. Mrhardy\thanks{n.mrhardy@ucam.ac.ma}\\
F.S.S.M, Universit\'{e} Cadi Ayyad, \\2390, Marrakech, Maroc}

\date{}
\maketitle \thispagestyle{empty} \setcounter{page}{1}

\begin{abstract}
This paper is intended to give a probabilistic representation for
stochastic viscosity solution of semi-linear reflected stochastic
partial differential equations with nonlinear Neumann boundary
condition. We use its connection with reflected generalized backward
doubly stochastic differential equations.
\end{abstract}


\vspace{.08in} \noindent \textbf{Keywords}: Backward doubly SDEs, Stochastic PDEs, Obstacle
problem, stochastic viscosity solutions.

\section{Introduction}
 Backward stochastic
differential equations (BSDEs, for short) were introduced by Pardoux
and Peng \cite{PP1} in 1990, and it was shown in various papers that
stochastic differential equations (SDEs) of this type give a
probabilistic representation for solution (at least in the viscosity
sence) of a large class of system of semi-linear parabolic partial
differential equations (PDEs). Thereafter a new class of BSDEs,
called backward doubly stochastic (BDSDEs), was considered by
Pardoux and Peng \cite{PP3}. The new kind of BSDEs seems suitable
for giving a probabilistic representation for a system of parabolic
stochastic partial differential equations (SPDEs). We refer to
Pardoux and Peng \cite{PP3} for the link between SPDEs
and BDSDEs in the particular case where  solutions of SPDEs are
regular. The more general situation is much more delicate to treat
because of the difficulties of extending the notion of viscosity solutions to SPDEs.

The notion of viscosity solution for PDEs was introduced by Crandall, Ishii
and Lions \cite{CL} for certain first-order Hamilton-Jacobi
equations. Today the theory has become an important tool in many
applied fields, especially in optimal control theory and numerous
subjects related to it.

The stochastic viscosity solution for semi-linear SPDEs was
introduced for the first time in Lions and Souganidis \cite{LS}.
They use the so-called "stochastic characteristic" to remove the
stochastic integrals from a SPDEs. On the other hand, two other ways
of defining a stochastic viscosity solution of SPDEs is considered by
Buckdahn and Ma respectively in \cite{BM1,BM2} and \cite{BM3}. In
the two first paper, they used the "Doss-Sussman"
 transformation to connect the stochastic viscosity solution of SPDEs with the solution of associated BDSDEs.
 In the second one, they introduced the stochastic viscosity solution by using the notion
 of stochastic sub and super jets. Next, in order to give a probabilistic
 representation for viscosity solution of  SPDEs with nonlinear Neumann boundary condition,
 Boufoussi et al. \cite{Bal} introduced the so-called generalized BDSDEs. They refer
 the first technique (Doss-Sussman transformation) of Buckdhan and Ma \cite{BM1,BM2}.

Based on the work of Boufoussi et al. \cite{Bal} and employing the penalized method from Ren et al. \cite{Ral}, the
aim of this paper, is to establish the existence result for semi-linear reflected SPDEs with
nonlinear Neumann boundary condition of the form:
\begin{eqnarray*}
\left\{
\begin{array}{l}
\min\left\{u(t,x)-h(t,x),\frac{\partial}{\partial t}u(t,x)-[{
L}u(t,x)-f(t,x,u(t,x),\sigma^{*}(x)\nabla u(t,x))]\right.\\\\
\left.\,\,\,\,\,\,\,\,\,\,\,\,\,\,\,\,\,\
-g(t,x,u(t,x))\lozenge B_{s}\right\}=0,\,\,\
(t,x)\in[0,T]\times\Theta\\\\ u(0,x)=l(x),\,\,\,\,\,\,\ x\in\overline{\Theta}\\\\
\frac{\displaystyle \partial u}{\displaystyle \partial
n}(t,x)+\phi(t,x,u(t,x))=0,\,\,\,\,\,\,\ x\in\partial\Theta,
\end{array}\right.
\end{eqnarray*}
where $\lozenge$ denotes the Wick product and, thus, indicates that the differential is to understand in It\^{o}'s sense.
Here $B$ is a standard Brownian motion, $L$ is an infinitesimal
generator of a diffusion process $X$, $\Theta$ is a connected bounded
domain and $f,\, g,\, \phi,\, l, h$ are some measurable
functions. More precisely, we give some direct links between the
stochastic viscosity solution of the previous reflected SPDE and the
solution of the following reflected generalized BDSDE:
\begin{eqnarray*}
Y_{t}&=&\xi+\int_{0 }^{t}f(s,Y_{s},Z_{s})ds+\int_{0
}^{t}\phi(s,Y_{s})dA_{s}+\int_{0}^{t}g(s,Y_{s})\,dB_{s}\\
&&-\int_{0}^{ t}Z_{s}\downarrow dW_{s}+K_{t},\,\ 0\leq t\leq
T.\label{a1}
\end{eqnarray*}
$\xi$ is the terminal value, $A$ is a positive real-valued
increasing process and $\downarrow dW_{s}$ denote the classical
backward It\^{o} integral with respect the Brownian motion $W$. Note
that our work can be considered as a generalization of two results.
First the one given in \cite{Ral}, where the authors treat
deterministic reflected PDEs with nonlinear Neumann boundary
conditions i.e $g\equiv0$. The second result appears in \cite{Bal} where
the non reflected SPDE with nonlinear Neumann
boundary condition is considered.

The present paper is organized as follows. An existence and
uniqueness result for solution to large class of reflected generalized BDSDEs is
shown in Section 2. Section 3 is devoted to give a definition of a
reflected stochastic solution to SPDEs and by the same occasion
establishes its existence result.

\section{Reflected generalized backward doubly stochastic differential equations}
\subsection{Notation, assumptions and definition.}

The scalar product of the space $\R^{d} (d\geq 2)$ will be denoted
by $<.,.>$ and the associated Euclidian norm  by $\|.\|$.

In what follows let us fix a positive real number
$T>0$. First of all $\{W_{t}, 0\leq t\leq T\}$ and $\{B_{t},\ 0\leq
t\leq T\}$ are two mutually independent standard Brownian motions
with values respectively in $\mbox{I\hspace{-.15em}R}^{d}$ and
$\mbox{I\hspace{-.15em}R}^{\ell}$, defined respectively on the two
probability spaces $(\Omega_{1},\mathcal{F}_{1},{\P}_{1})$ and
$(\Omega_{2},\mathcal{F}_{2},{\P}_{2})$. Let ${\bf
F}^{B}=\{\mathcal{F}^{B}_t\}_{t\geq 0}$ denote the natural
filtration generated by $B$, augmented by the $\P_{1}$-null sets of
$\mathcal{F}_{1}$; and let
$\mathcal{F}^{B}=\mathcal{F}^{B}_{\infty}$. On the other hand we
consider the following family of $\sigma$-fields:
\begin{eqnarray*}
\mathcal{F}^{W}_{t,T}=\sigma\{W_{s}-W_{T},t\leq s\leq T\}\vee
\mathcal{N}_{2},
\end{eqnarray*}
where $\mathcal{N}_{2}$ denotes all the $\P_{2}$- null sets in
$\mathcal{F}_{2}$. We also denote
${\bf F}^{W}_{T}=\{\mathcal{F}^{W}_{t,T}\}_{0\leq t\leq T}$.

Next
we consider the product space $(\Omega,\mathcal{F},\P)$ where
\begin{eqnarray*}
\Omega=\Omega_{1}\times\Omega_{2},\,\,\mathcal{F}=\mathcal{F}_{1}\otimes\mathcal{F}_{2}\,\,
\mbox{ and}\,\, \P=\P_{1}\otimes\P_{2}.
\end{eqnarray*}
For each $t\in[0,T]$, we define
\begin{eqnarray*}
\mathcal{F}_{t}=\mathcal{F}_{t}^{B}\otimes\mathcal{F}^{W}_{t,T} .
\end{eqnarray*}

Note that the collection ${\bf F}= \{\mathcal{F}_{t},\ t\in
[0,T]\}$ is neither increasing nor decreasing and it does not
constitute a filtration.

Further, we assume that random variables
$\xi(\omega_{1}),\;\omega_{1}\in \Omega_{1}$ and
$\zeta(\omega_{2}),\; \omega_{2}\in \Omega_{2}$ are considered as
random variables on $\Omega $ via the following identification:
\begin{eqnarray*}
\xi(\omega_{1},\omega_{2})=\xi(\omega_{1});\,\,\,\,\,\zeta(\omega_{1},\omega_{2})=\zeta(\omega_{2}).
\end{eqnarray*}

In the sequel, let\
$\{A_t,\ 0\leq t\leq T\}$\ be a continuous, increasing and
${\bf F}$-adapted real valued  process such that\ $A_0=0.$

For
any $d\geq 1$, we consider the following spaces of processes:
\begin{enumerate}
\item $M^{2}(0,T,\R^{d})$ denote the Banach space of all
equivalence classes (with respect to the measure $d\mathbb{P}\times
dt$) where each equivalence class contains an d-dimensional jointly
measurable stochastic process $\displaystyle{\varphi_{t};
t\in[0,T]}$, which satisfies :
\begin{description}
\item $(i)$ $\displaystyle{\|\varphi\|^{2}_{M^{2}}=\E\int ^{T}_{0}|\varphi_{t}|^{2}dt<\infty}$;

\item $(ii)$ $\varphi_t$ is ${\mathcal{F}}_{t}$-measurable , for any $t\in [0,T]$.
\end{description}
\item  $S^{2}([0,T],\R)$ is the set of one
dimensional continuous stochastic processes which verify:
\begin{description}
\item $(iii)$ $\displaystyle{\|\varphi\|^{2}_{S^{2}}=\E\left(\sup_{0\leq t\leq T}|\varphi_{t}|
^{2}+\int^T_0|\varphi_s|^{2}dA_s\right)<\infty}$;

\item $(iv)$ $\varphi_t$ is ${\mathcal{F}}_{t}$-measurable , for any $t\in [0,T]$.
\end{description}
\end{enumerate}
Let us give the data $(\xi,f,g,\phi, S)$ which satisfy:
\begin{itemize}
\item[$(\textbf{H}_1)$] $\xi$ is a square integrable random variable which is $\mathcal{F}_{T}$-measurable such  that for all $\mu>0$
$$ \mathbb{E}\left(e^{\mu A_T}|\xi|^2\right) < \infty. $$
\item[$(\textbf{H}_2)$]
$f:\Omega\times [0,T]\times\R\times\R^{d}\rightarrow \R$,\
$g:\Omega\times [0,T]\times\R\times\R^{d}\rightarrow \R^{\ell},$\
and $\phi:\Omega\times [0,T]\times\R \rightarrow \R,$ are three
functions such that:
\begin{itemize}
  \item[$(a)$] There exist $\mathcal{F}_t$-adapted processes $\{f_t,\,
\phi_t,\,g_t:\,0\leq t\leq T\}$ with values in $[1,+\infty)$ and
with the property that for any $(t,y,z)\in
[0,T]\times\R\times\R^{d}$, and any $\mu>0$, the following hypotheses
are satisfied for some strictly positive finite constant $K$:
\begin{eqnarray*}
\left\{
\begin{array}{l}
f(t,y,z),\, \phi(t,y),\,\mbox{and}\, g(t,y,z)\, \mbox{are}\, \mathcal{F}_t\mbox{-measurable processes},\\\\
|f(t,y,z)|\leq f_t+K(|y|+\|z\|),\\\\
|\phi(t,y)|\leq \phi_t+K|y|,\\\\
|g(t,y,z)|\leq g_t+K(|y|+\|z\|),\\\\
\displaystyle \E\left(\int^{T}_{0} e^{\mu
A_t}f_t^{2}dt+\int^{T}_{0}e^{\mu A_t}g_t^{2}dt+\int^{T}_{0}e^{\mu
A_t}\phi_t^{2}dA_t\right)<\infty.
\end{array}\right.
\end{eqnarray*}
\item[$(b)$] There exist constants $c>0,\, \beta<0$ and $0<\alpha<1$ such that for any $(y_1,z_1),\,(y_2,z_2)\in\R\times\R^{d}$,
\begin{eqnarray*}
\left\{
\begin{array}{l}
(i)\, |f(t,y_1,z_1)-f(t,y_2,z_2)|^{2}\leq c(|y_1-y_2|^{2}+\|z_1-z_2\|^{2}),\\\\
(ii)\, |g(t,y_1,z_1)-g(t,y_2,z_2)|^{2}\leq c|y_1-y_2|^{2}+\alpha\|z_1-z_2\|^{2},\\\\
(iii)\; \langle y_1-y_2,\phi(t,y_1)-\phi(t,y_2)\rangle\leq
\beta|y_1-y_2|^{2}.
\end{array}\right.
\end{eqnarray*}
\end{itemize}
\item[($\textbf{H}_3$)] The obstacle $\left\{  S_{t},0\leq t\leq T\right\}$,
is a continuous $\mathcal{F}_{t}$-progressively measurable
real-valued process satisfying for any $\mu>0$ $$\E\left(  \sup_{0\leq t\leq T}e^{\mu A_t}\left|
S^+_{t}\right| ^{2}\right)  <\infty.$$ We shall always assume that
$S_{T}\leq\xi\ a.s.$
\end{itemize}
One of our main goal in this paper is the study of reflected generalized BDSDEs,
\begin{eqnarray}
Y_{t}&=&\xi+\int_{0 }^{t}f(s,Y_{s},Z_{s})ds+\int_{0
}^{t}\phi(s,Y_{s})dA_s+\int_{0}^{t}g(s,Y_{s},Z_{s})\,dB_{s}\nonumber\\
&&-\int_{0}^{ t}Z_{s}\downarrow dW_{s}+K_{t},\,\ 0\leq t\leq
T.\label{a1}
\end{eqnarray}
First of all let us give a definition to the solution of this BDSDEs.
\begin{definition}
By a solution of the reflected generalized BDSDE $(\xi,f,\phi,g,S)$
we mean a triplet of processes  $(Y,Z,K)$, which satisfies (\ref{a1})
such that the following holds $\P$- a.s
\begin{description}
\item $(i)$ $(Y,Z)\in S^{2}([0,T];\R)\times M^{2}(0,T;\R^{d})$
\item  $(ii)$ the map $s\mapsto Y_s$ is continuous
\item  $(iii)$ $Y_{t}\geq S_{t},\,\,\,\ 0\leq t\leq T$,\,\
\item  $(iv)$ $K$ is an increasing process such that $K_{0}=0$\ %
and\ $\displaystyle \int_{0}^{ T}\left( Y_{t}-S_{t}\right)
dK_{t}=0$.
\end{description}
\end{definition}
\begin{remark}
We note that although the equation ($\ref{a1}$) looks like a forward
SDE, it is indeed a backward one because a terminal condition is
given at $t=0$ ($Y_{0}=\xi$). We use this  technique of reversal time due
to the set-up of our problem that is, its connection to the the form of our obstacle
problem for SPDE with nonlinear Neumann boundary condition.
\end{remark}
In the sequel, $C$ denotes a positive constant which may vary from
one line the other.
\subsection{Comparison theorem}
Let us give this comparison theorem related of the generalized
BDSDE, which we will need in the proof of our main result. The proof
follows with the same computation as in \cite{Yal}, with slight
modification due to the presence of the integral with respect the
increasing process $A$. So we just repeat the main step.
\begin{theorem}(Comparison theorem for generalized BDSDE)\label{Thm:comp}
Let $(Y,Z)$ and $(Y',Z')$ be the unique solution
of the non reflected generalized BDSDE associated to
$(\xi,f,\phi,g)$ and $(\xi',f',\phi,g)$ respectively. If $\xi\leq
\xi',\; f(t,Y'_t,Z'_t)\leq f'(t,Y'_t,Z'_t)$ and $\phi(t,Y'_t)\leq
\phi^{'}(t,Y'_t)$, then $Y_t\leq Y'_t,\; \forall\, t\in [0, T]$.
\end{theorem}
\begin{proof}
Let us set $\Delta Y=Y-Y'$, $\Delta Z=Z-Z'$ and $(\Delta Y)^{+}=(Y-Y')^{+}$ (with $f^{+}=\sup\{f,0\}$).\newline Using It\^o's formula, we get for all $0\leq t\leq T$
\begin{align}\label{eq:1}
&\mathbb{E}((\Delta Y_t)^+)^2+\mathbb{E}\int_0^t \|\Delta
Z_s\|^2{\bf 1}_{\{Y_s>Y'_s\}}ds
\nonumber\\
\leq &\E((\xi-\xi')^{+})^2+ 2\mathbb{E}\int_0^t
(\Delta Y_s)^+{\bf 1}_{\{Y_s>Y'_s\}}\left\{f(s,Y_s,Z_s)-f'(s,Y'_s,Z'_s)\right\}ds\nonumber\\
&+2\mathbb{E}\int_0^t
(\Delta Y_s)^+{\bf 1}_{\{Y_s>Y'_s\}}\left\{\phi(s,Y_s)-\phi'(s,Y'_s)\right\}dA_s\nonumber\\
 &+\mathbb{E}\int_0^t\left\|g(s,Y_s,Z_s)-g(s,Y'_s,Z'_s)\right\|^2{\bf 1}_{\{Y_s>Y'_s\}}ds,
\end{align}
where ${\bf 1}_{\Gamma}$ denote the characteristic function of a
given set $\Gamma\in {\bf F}$ defined by

${\bf 1}_{\Gamma}(\omega)=\left\{
\begin{array}{l}
1\ \mbox{if}\ \omega\in \Gamma\\
0\ \mbox{if}\ \omega\in \Gamma.
\end{array}\right.
$\newline
From $(\textbf{H}_2)(b)$ we have
\begin{eqnarray*}
2(\Delta Y_s)^+\left\{f(s,Y_s,Z_s)-f'(s,Y'_s,Z'_s)\right\}
&\leq & 2(\Delta Y_s)^+\left\{f(s,Y_s,Z_s)-f(s,Y'_s,Z'_s)\right\}\\
&\leq &(\frac{1}{\varepsilon}+\varepsilon c)((\Delta
Y_s)^+)^2+\varepsilon c\|\Delta Z_s\|^2,
\end{eqnarray*}
\begin{eqnarray*}
2(\Delta Y_s)^+\left\{\phi(s,Y_s)-\phi'(s,Y'_s)\right\} &\leq &
2(\Delta Y_s)^+\left\{\phi(s,Y_s)-\phi(s,Y'_s)\right\}\\
&\leq & \beta((\Delta Y_s)^+)^2
\end{eqnarray*}

and
\begin{eqnarray*}
\left\|g(s,Y_s,Z_s)-g(s,Y'_s,Z'_s)\right\|^2{\bf 1}_{\{Y_s>Y'_s\}}
&\leq&c((\Delta Y_s)^+)^2{\bf 1}_{\{Y_s>Y'_s\}}+\alpha\|\Delta
Z_s\|^2{\bf 1}_{\{Y_s>Y'_s\}}.
\end{eqnarray*}
Plugging these inequalities on (\ref{eq:1}) and choosing
$\displaystyle \varepsilon=\frac{1-\alpha}{2c}$, we conclude that
$$ \mathbb{E}((\Delta Y_t)^+)^2\leq 0$$
which leads to $\Delta Y^+_t=0$ a.s. and so $Y'_t\geq Y_t$ a.s. for
all $t\leq T$.
\end{proof}
\subsection {Existence and Uniqueness result}
Our main goal in this section is to prove the following theorem.
\begin{theorem}\label{Th:exis-uniq}
Under the hypotheses $(\textbf{H}_1),\ (\textbf{H}_2)$ and
$(\textbf{H}_3)$, there exists a unique solution for the reflected
generalized BDSDE $(\xi,f,\phi,g,S)$.
\end{theorem}

Our proof is based on a penalization method but is slightly
different from El Karoui et al \cite{Kal}, because of the presence
of the two integral with respect the increasing process $A$ and the
Brownian motion $B$, and also because of the time reversal.

For each $n\in\N^{*}$, we set
\begin{eqnarray}\label{a0}
f_{n}(s,y,z)=f(s,y,z)+n(y-S_{s})^{-}
\end{eqnarray}
and consider the generalized BDSDE
\begin{eqnarray}
Y_{t}^{n}&=&\xi+\int_{0}^{t}f_{n}(s,Y_{s}^{n},Z_{s}^{n})ds+\int_{0}^{t}\phi(s,Y_{s}^{n})dA_s\nonumber\\
&&+\int_{0}^{t}g(s,Y_{s}^{n},Z^{n}_{s})\,dB_{s}-\int_{0}^{t}Z_{s}^{n}\downarrow
dW_{s}, \label{h2}
\end{eqnarray}
obtained by the penalized method. We point out that the previous version of generalized
BDSDE is, in fact, the time reversal version of that considered in Boufoussi et
al \cite{Bal}, due to the set-up of our problem. We nonetheless use
the same name because they are similar in nature. Consequently, it is
well known (see Boufoussi et al., \cite{Bal}) that, there exist a
unique $(Y^{n},Z^{n})\in S^{2}([0,T];\R)\times M^{2}(0,T;\R^{d})$
solution of the generalized BDSDE $(\ref{h2})$ such that for each
$n\in\N^{*}$,
\begin{eqnarray*}
\E\left(\sup_{0\leq t\leq
T}|Y^{n}_{t}|^{2}+\int_0^T\left\|Z_s^{n}\right\|^2 ds
\right)<\infty.
\end{eqnarray*}
In order to prove Theorem \ref{Th:exis-uniq}, we state the following
lemmas that will be useful.
\begin{lemma}\label{lem:1}
Let us consider  $(Y^{n},Z^{n})\in S^{2}([0,T];\R)\times
M^{2}(0,T;\R^{d})$ solution of BDSDE $(\ref{h2})$. Then for any $\mu>0$, there exists
$C>0$ such that,
\begin{eqnarray*}
\sup_{n\in\N^{*}}\E\left( \sup_{0\leq t\leq T}e^{\mu A_t}\left|
Y_{t}^{n}\right| ^{2}+\int_{0}^{T}e^{\mu A_s}\left| Y_{s}^{n}\right| ^{2}
dA_s+\int_{0}^{T}e^{\mu A_s}\left\| Z_{s}^{n}\right\| ^{2}
ds+|K^{n}_{T}|^{2}\right)<C
\end{eqnarray*}
where
\begin{eqnarray} K_{t
}^{n}=n\int_{0}^{t}(Y_{s}^{n}-S_{s})^{-}ds , \,\ 0\leq t\leq
T.\label{K}
\end{eqnarray}
\end{lemma}
\begin{proof}
From It\^{o}'s formula, it follows that
\begin{align}\label{b2}
& e^{\mu A_t}\left| Y_{t}^{n}\right|^{2}+\int_{0}^{t}e^{\mu A_s}\left\| Z_{s}^{n}\right\|
^{2}ds \nonumber\\
& \leq e^{\mu A_T}\left| \xi\right| ^{2}+2\int_{0}^{t}e^{\mu
A_s}Y_{s}^{n}f(s,Y_{s}^{n},Z_{s}^{n}) ds
+2\int_{0}^{t}e^{\mu A_s}Y_{s}^{n}\phi(s,Y_{s}^{n})dA_s-\mu\int_{0}^{t}e^{\mu A_s}|Y_{s}^{n}|^2dA_s\nonumber\\
& +\int_{0}^{t}e^{\mu
A_s}\|g(s,Y_{s}^{n},Z_{s}^{n})\|^{2}ds+2\int_{0}^{t} e^{\mu
A_s}S_{s}dK_{s}^{n}+ 2\int_{0}^{t}e^{\mu A_s}\langle Y_{s}^{n}
,g(s,Y_{s}^{n},Z_{s}^{n})dB_{s}\rangle\nonumber\\
& -2\int_{0}^{t}e^{\mu A_s}\langle
Y_{s}^{n},Z_{s}^{n}\downarrow dW_{s}\rangle,
\end{align}
where we have used
$\displaystyle{\int_{0}^{t}e^{\mu A_s}(Y^{n}_{s}-S_{s})dK^{n}_{s}\leq 0}$ and
the fact that
\begin{eqnarray*}
\int_{0}^{t}e^{\mu A_s}Y^{n}_{s}dK^{n}_{s}&=&
\int_{0}^{t}e^{\mu A_s}(Y^{n}_{s}-S_{s})dK^{n}_{s}+\int_{0}^{t}e^{\mu A_s}S_{s}dK^{n}_{s}\leq
\int_{0}^{t}e^{\mu A_s}S_{s}dK^{n}_{s}.
\end{eqnarray*}
Using $(\textbf{H}_2)$ and the elementary inequality $2ab\leq \gamma
a^{2}+\frac{1}{\gamma} b^2,\ \forall\gamma>0$,
\begin{eqnarray*}
2Y_{s}^{n}f(s,Y_{s}^{n},Z_{s}^{n})&\leq&(c\gamma_1+\frac{1}{\gamma_1})|Y^{n}_{s}|^{2}
+2c\gamma_1\|Z^{n}_{s}\|^{2}+2\gamma_1 f_s^{2},\\
2Y_{s}^{n}\phi(s,Y_{s}^{n})&\leq& (\gamma_2-2|\beta|-\mu)|Y^{n}_{s}|^{2}+\frac{1}{\gamma_2}\phi_s^{2},\\
\|g(s,Y_{s}^{n},Z_{s}^{n})\|^{2}
&\leq&(1+\gamma_3)c|Y^{n}_{s}|^{2}+\alpha(1+\gamma_3)\|Z^{n}_{s}\|^{2}+(\frac{1}{\gamma_3}+1)
g_s^{2}.
\end{eqnarray*}
Taking expectation in both sides of the inequality (\ref{b2}) and
choosing $\displaystyle\gamma_1=\frac{1-\alpha}{6c}$,
$\displaystyle\gamma_2-\mu=|\beta|$ and
$\displaystyle\gamma_3=\frac{1-\alpha}{2\alpha}$ we obtain for all
$\varepsilon >0$
\begin{align}\label{a2}
&
\E (e^{\mu A_t}\left|Y_{t}^{n}\right|^{2})+|\beta|\E\int_{0}^{t}e^{\mu A_s}\left|Y_{s}^{n}\right|^{2}dA_s
+\frac{1-\alpha}{6}\E\int_{0}^{t}e^{\mu A_s}\left\|Z_{s}^{n}\right\|^{2}ds \nonumber\\
&\leq
C\E\left\{e^{\mu A_T}|\xi|^{2}+\int^{t}_{0}e^{\mu A_s}|Y^{n}_{s}|^{2}ds+\int^{t}_{0}e^{\mu A_s}f_s^{2}ds+
\int^{t}_{0}e^{\mu A_s}\phi_s^{2}dA_s+\int^{t}_{0}e^{\mu A_s}g_s^{2}ds\right\}\nonumber\\
& +\frac{1}{\varepsilon}\E\left(\sup_{0\leq s\leq
t}(e^{\mu A_s}S_{s}^+)^2\right)+\varepsilon\E \left(K_{t}^{n}\right)^2.
\end{align}
On the other hand, we get from (\ref{h2}) that for all $0 \leq t
\leq T$,
\begin{equation}\label{K:n}
K_t^n=Y_t^n-\xi-\int_0^t f(s,Y_s^n,Z_s^n) ds -\int_0^t
\phi(s,Y_s^n) dA_s-\int_0^t g(s,Y_s^n,Z_s^n) dB_s+\int_0^t Z_s^n
\downarrow d W_s.
\end{equation}
Then we have
\begin{align}\label{estK}
\E(K^n_t)^2\leq 5\E\left\{e^{\mu A_T}|\xi|^{2}+e^{\mu A_t}|Y^{n}_t|^2+\left|\int^{t}_{0}f(s,Y^{n}_s,Z^n_s)ds\right|^{2}\right.\nonumber\\
\left.+\left|\int^{t}_{0}\phi(s,Y^n_s)dA_s\right|^{2}+\left|\int^{t}_{0}g(s,Y^{n}_s,Z^{n}_s)dB_s\right|^{2}+\left|\int^{t}_{0}Z^{n}_s\downarrow dW_s\right|^{2}\right\}.
\end{align}
It follows by H\"{o}lder inequality and the isometry equality, together with assumptions $({\bf H_{2}})(a)$ that
\begin{eqnarray*}
\left|\int^{t}_{0}f(s,Y^{n}_s,Z^n_s)ds\right|^{2}\leq
3\int^{t}_{0}e^{\mu
A_s}(f_s^2+K^2|Y^{n}_s|^{2}+K^2\|Z^{n}_s\|^{2})ds,
\end{eqnarray*}
\begin{eqnarray*}
\E\left|\int^{t}_{0}g(s,Y^{n}_s,Z^{n}_s)dB_s\right|^{2}\leq
3\E\int^{t}_{0}e^{\mu A_s}[g_s^{2}+K^2|Y^{n}_{s}|^{2}+K^2\|Z^{n}_{s}\|^{2}]ds.
\end{eqnarray*}
and
\begin{eqnarray*}
\E\left|\int^{t}_{0}Z^{n}_s\downarrow dW_s\right|^{2}\leq \E\int^{t}_{0}e^{\mu A_s}|Z^{n}_s|^{2}ds.
\end{eqnarray*}
Next, to estimate  $\left|\int^{t}_{0}\phi(s,Y^n_s)dA_s\right|^{2}$,
let us assume first that $A_T$ is a bounded  real variable. For any $\mu>0$ given in assumptions $({\bf
H}_1)$ or $({\bf H}_2)(a)$, we have
\begin{eqnarray*}
\left|\int^{t}_{0}\phi(s,Y^{n}_s)dA_s\right|^{2}&\leq& \left(\int^{t}_{0}e^{-\mu A_s}dA_s\right)\left(\int^{t}_{0}e^{\mu A_s}|\phi(s,Y^{n}_s)|^{2}dAs\right)\nonumber\\
&\leq& \frac{2}{\mu}\int^{t}_{0}e^{\mu A_s}(\phi_s^2+K^2|Y^{n}_s|^{2})dA_s,\label{psi}
\end{eqnarray*}
since
\begin{eqnarray*}
\left(\int^{t}_{0}e^{-\mu A_s}dA_s\right)\leq \frac{1}{\mu}[1-e^{-\mu A_T}]\leq \frac{1}{\mu}.
\end{eqnarray*}
The general case then follows from Fatou's lemma.

Therefore, from $(\ref{estK})$ together with the previous
inequalities, there exists a constant independent of $A_T$ such that
\begin{align}\label{kn1}
\E(K^n_t)^2&\leq C\E\left\{e^{\mu A_T}|\xi|^{2}+e^{\mu
A_t}|Y^{n}_t|^2+\int^{t}_{0}e^{\mu A_s}f_s^{2}ds+ \int^{t}_{0}e^{\mu
A_s}\phi_s^{2}dA_s+\int^{t}_{0}e^{\mu A_s}g_s^{2}ds\right.\nonumber\\
&\left.+\int_0^t e^{\mu
A_s}\left|Y_s^n\right|^2ds+\E\left(\sup_{0\leq s\leq t}e^{\mu
A_s}(S_{s}^+)^2\right)+\int_0^t e^{\mu A_s}
\left|Y_s^n\right|^2dA_s+\int_0^t e^{\mu A_s}\|Z_s^n\|^2ds\right\}.
\end{align}
Recalling again $(\ref{a2})$ and taking $\varepsilon$ small enough
such that $\varepsilon C<\min\{1,|\beta|,\frac{1-\alpha}{6}\}$, we
obtain
\begin{align*}
&
\E e^{\mu A_t}\left|Y_{t}^{n}\right|^{2}+\E\int_{0}^{t}e^{\mu A_s}\left|Y_{s}^{n}\right|^{2}dA_s
+\E\int_{0}^{t}e^{\mu A_s}\left\|Z_{s}^{n}\right\|^{2}ds \\
&\leq
C\E\left\{e^{\mu A_T}|\xi|^{2}+\int^{t}_{0}e^{\mu A_s}|Y^{n}_{s}|^{2}ds+\int^{t}_{0}e^{\mu A_s}f_s^{2}ds+
\int^{t}_{0}e^{\mu A_s}\phi_s^{2}dA_s\right.\\
&\left.+\int^{t}_{0}e^{\mu A_s}g_s^{2}ds
+\E\left(\sup_{0\leq s\leq T}e^{\mu A_s}(S_{s}^+)^2\right)\right\}
\end{align*}
Consequently, it follows from Gronwall's lemma and (\ref{kn1}) that
\begin{align*}
&
\E\left\{e^{\mu A_t}|Y_{t}^{n}|^{2}+\int_{0}^{t}e^{\mu A_s}\left|Y_{s}^{n}\right|^{2}dA_s
+\int_{0}^{t}e^{\mu A_s}\|Z_{s}^{n}\|^{2}ds+|K_{T}^{n}|^{2}\right\}\\
&\leq
C\E\left\{e^{\mu A_T}|\xi|^{2}+\int^{T}_{0}e^{\mu A_s}f_s^{2}ds+\int^{T}_{0}e^{\mu A_s}\phi_s^{2}dA_s
+\int^{T}_{0}e^{\mu A_s}g_s^{2}ds+\sup_{0\leq t\leq T}e^{\mu A_t}(S_{t}^{+})^{2}\right\}.
\end{align*}
Finally, by  application of Burkholder-Davis-Gundy inequality we obtain from $(\ref{b2})$
\begin{eqnarray*}
\E\left\{\sup_{0\leq t\leq T}
e^{\mu A_t}|Y_{t}^{n}|^{2}+\int_{0}^{T}e^{\mu A_s}\|Z_{s}^{n}\|^{2}ds+|K_{T}^{n}|^{2}\right\}
&\leq & C\E\left\{e^{\mu A_T}|\xi|^{2}+\int^{T}_{0}e^{\mu A_s}f_s^{2}ds+\int^{T}_{0}e^{\mu A_s}\phi_s^{2}dA_s\right.\nonumber\\
&&+\left.\int^{T}_{0}e^{\mu A_s}g_s^{2}ds+\sup_{0\leq t\leq
T}e^{\mu A_t}(S_{t}^{+})^{2}\right\},
\end{eqnarray*}
which end the proof of this Lemma.
\end{proof}

Now we give a convergence result which is the key point on the proof
of our main result. We begin by supposing that  $g $ is independent
from $\left( Y,Z\right)$. More precisely, we consider the following
equation
\begin{eqnarray}
Y_{t}&=&\xi+\int_{0}^{t}f(s,Y_s,Z_s)ds+
\int_{0}^{t}\phi(s,Y_s)dA_s+\int_{0}^{t}g(s)\,dB_{s}-\int_{0}^{t}Z_{s}\downarrow
dW_{s}+K_t. \nonumber\\
&&\label{h4}
\end{eqnarray}
The penalized equation is given by
\begin{eqnarray}
Y_{t}^{n}&=&\xi+\int_{0}^{t}f(s,Y_s^n,Z_s^n)ds+n\int_{0}^{t}(Y_{s}^{n}-S_{s})^{-}ds+
\int_{0}^{t}\phi(s,Y_{s}^{n})dA_s\nonumber\\
&&+\int_{0}^{t}g(s)\,dB_{s}-\int_{0}^{t}Z_{s}^{n}\downarrow dW_{s}.
\label{h3}
\end{eqnarray}
\noindent Since the sequence of functions $(y\mapsto
n(y-S_{t})^{-})_{n\geq 1}$ is nondecreasing, then  thanks to the
comparison theorem \ref{Thm:comp}, the sequence $\left( Y^{n}\right)
_{n>0}$ is non-decreasing. Hence, Lemma \ref{lem:1} implies that
there exists a $\mathcal{F}_t$- progressively measurable process $Y$
such that $Y_{t}^{n}\nearrow Y_{t}$ $a.s$. So the following result
holds.
\begin{lemma}\label{lem:2}
If $g $ does not dependent on $\left( Y,Z\right)$, then for
each $n\in\N^{*}$,
\begin{eqnarray*}
\E\left(\sup_{0\leq t\leq T }\left|\left( Y_{t}^{n}-S_{t}\right)
^{-}\right| ^{2}\right)\longrightarrow 0,\,\,\ \mbox{as}\,\ n
\longrightarrow \infty.
\end{eqnarray*}
\end{lemma}
\begin{proof}
Since $Y^{n}_{t}\geq Y^{0}_{t}$, we can w.l.o.g. replace $S_{t}$ by
$S_{t}\vee Y^{0}_{t}$, i.e. we may assume that\\ $\E(\sup_{0\leq
t\leq T}S_{t}^{2})<\infty$. We want to compare a.s. $Y_t$ and $S_t$
for all $t\in[0,T]$, while we do not know yet if $Y$ is a.s.
continuous. Indeed, let us introduce the following processes
$$
\left\{
\begin{array}{ll}
&\displaystyle \overline{\xi}:=\xi+\int_{0}^{T}g\left(
s\right)  dB_{s}\\
& \displaystyle\overline{S}_{t}:=S_{t}+\int_{t}^{T}g\left(
s\right)  dB_{s}\\
& \displaystyle\overline{Y}_{t}^{n}:=Y_{t}^{n}+\int_{t} ^{T}g\left(
s\right)  dB_{s}
\end{array}
\right.
$$
Hence,  \begin{equation}\label{Ynbar}
\overline{Y}_{t}^{n}=\overline{\xi}+\int_{0}^{t}f\left(
s,Y_s^n,Z_s^n\right) ds+n\int_{0}^{t}\left(\overline{Y}_{s}^{n}
-\overline{S_{s} }\right) ^{-}ds+\int_{0}^{t}\phi\left(s,
Y_s^n\right) dA_s-\int_{0}^{t}Z_{s}^{n} \downarrow dW_{s}.
\end{equation}
and we define \quad
$\overline Y_t:=\displaystyle\sup_n \overline Y^n_t$.\\
From Theorem \ref{Thm:comp}, we have that a.s.,
$\overline{Y}^{n}_t\geq\widetilde{Y}_{t}^{n},\, 0\leq t\leq T, \,
n\in\N^{*},$ where $\left\{
(\widetilde{Y_{t}}^{n},\widetilde{Z}_{t}^{n}),\mbox{ }0\leq t\leq
T\right\} $ is the unique solution of the  BSDE
\begin{eqnarray*}
\widetilde{Y}_{t}^{n} &=&\overline{S}_{T}+\int_{0}^{t}f\left(
s,Y_s^n,Z_s^n\right)
ds+n\int_{0}^{t}(\overline{S}_{s}-\widetilde{Y}_{s}^{n})ds
+\int_{0}^{t}\phi\left(s, Y_s^n\right)
dA_s-\int_{0}^{t}\widetilde{Z} _{s}^{n}\downarrow dW_{s}.
\end{eqnarray*}
Let ${\bf G}=(\mathcal{G}_{t})_{0\leq t\leq T}$ be a filtration defined by $\mathcal{G}_{t}=\mathcal{F}^{W}_{t,T}\otimes\mathcal{F}^{B}_{0,T}$. We consider $\nu $
a ${\bf G}$-stopping time such that $0\leq \nu \leq T$.
So we can write
\begin{eqnarray}
\widetilde{Y}_{\nu }^{n} &=&\E\left\{ e^{-n\nu}\overline{S}_{T}
+\int_{ 0}^{\nu }e^{-n(\nu-s)}f(s,Y_{s}^{n},Z_{s}^{n})ds+n\int_{
0}^{\nu }e^{-n(\nu-s)}\overline{S}_{s}ds \right.\nonumber \\
&&\left.+\int_{ 0}^{\nu }e^{-n(\nu-s)}\phi(s,Y_{s}^{n})dA_s\mid
{\cal G} _{\nu }\right\}. \label{c'2}
\end{eqnarray}
First, with the help of H\"{o}lder inequality and assumptions $({\bf
H}_2)(a)$, we have
 \begin{eqnarray*}
\E\left(\int_{0}^{\nu}e^{-n(\nu-s)}f(s,Y_{s}^{n},Z_{s}^{n})ds\right)^{2}&\leq&
\frac{1}{2n}\E\left( \int_{0 }^{\nu}|
f(s,Y_{s}^{n},Z_{s}^{n}|^{2} ds\right)\\
&\leq& \frac{C}{2n}\E\left( \int_{0 }^{T}
e^{\mu A_s}(f_s^{2}+|Y_{s}^{n}|^{2}+\|Z_{s}^{n}\|^{2}) ds\right),
\end{eqnarray*}
which provide
\begin{eqnarray}
\E\left(\int_{0}^{\nu}e^{-n(\nu-s)}f(s,Y_{s}^{n},Z_{s}^{n})ds\right)^{2}\longrightarrow
0 \,\;\mbox{as}\,\; n\rightarrow \infty,\label{ds}
\end{eqnarray}
since $\E\left( \int_{0 }^{T} e^{\mu
A_s}(f_s^{2}+|Y_{s}^{n}|^{2}+\|Z_{s}^{n}\|^{2}) ds\right)<C$ (see
Lemma 2.1 and $({\bf H}_2)(a)$).

Next, to prove that
\begin{eqnarray}
\E\left(\int_{0}^{\nu}e^{-n(\nu-s)}\phi(s,Y_{s}^{n})dA_s\right)^{2}\longrightarrow
0 \,\;\mbox{as}\;\;n\rightarrow \infty,\label{A}
\end{eqnarray}
let first suppose that there exists $C_1$ such that $\|A_T\|_{\infty}<C_1$. Using again H\"{o}lder inequality, Lemma 2.1 and assumption $({\bf H}_2)(a)$, we
get
\begin{eqnarray*}
\E\left(\int_{0}^{\nu}e^{-n(\nu-s)}\phi(s,Y_{s}^{n})dA_s\right)^{2}
&\leq& \E\left[\left(\int_{0}^{T}e^{-[2n(\nu-s)+\mu A_{s}]}dA_s\right)\left(
\int_{0 }^{T}e^{\mu A_s}| \phi(s,Y_{s}^{n})|^{2} dA_s\right)\right]\\
&\leq &\frac{1}{\mu}(1-e^{-\mu C_1})\E\left(
\int_{0 }^{T}e^{\mu A_s}(\phi_s^2+K|Y_{s}^{n}|^{2}) dA_s\right)\\
&\leq & C
\end{eqnarray*}
where $C$ is independent of $A_T$. The result follows by Lebesgue
dominated Theorem, since
$\int_{0}^{\nu}e^{-n(\nu-s)}\phi(s,Y_{s}^{n})dA_s\rightarrow 0 \,\,
a.s.\,\mbox{as}\,\, n\rightarrow \infty$. On the other hand it is
easily seen that
\begin{eqnarray}
e^{-n\nu}\overline{S}_{T}+n\int_{0}^{\nu }e^{-n(\nu-s)}\overline{S}_{s}ds\rightarrow \overline{S}_{\nu
}{\bf 1}_{\{\nu >0\}}+\overline{S}_{T}{\bf 1}_{\{\nu =0\}}\,\, a.s.\,\mbox{as}\,\, n\rightarrow \infty.\label{S}
\end{eqnarray}
According to $(\ref{ds})$-$(\ref{S})$, the equality $(\ref{c'2})$
provides
$$\widetilde{Y}_{\nu }^{n}\longrightarrow \overline{S}_{\nu
}{\bf 1}_{\{\nu >0\}}+\overline{S}_{T}{\bf 1}_{\{\nu =0\}}\ a.s.$$ and in $L^{2}(\Omega)$, as $n\rightarrow \infty$, and
$\overline{Y}_{\nu}\geq \overline{S}_{\nu}$ a.s. which yields that
${Y}_{\nu}\geq {S}_{\nu}$ a.s. From this and the Section Theorem in
Dellacherie and Meyer \cite{DM}, it follows that the last inequality holds for all
$t\in[0,T]$. Further $(Y_{t}^{n}-S_{t})^{-}\downarrow 0 $, a.s. and
from Dini's theorem, the convergence is uniform in $t$. Finally, as
$\displaystyle{(Y_{t}^{n}-S_{t})^{-}\leq (S_{t}-Y_{t}^{0})^{+}\leq
\left| S_{t}\right| +\left| Y_{t}^{0}\right|}$, the dominated
convergence theorem ensures that
\begin{eqnarray*}
\lim_{n\longrightarrow +\infty }\E(\sup_{0\leq t\leq T }|\left(
Y_{t}^{n}-S_{t}\right) ^{-}| ^{2})=0.
\end{eqnarray*}
\end{proof}

\noindent \begin{proof}[Proof of Theorem \ref{Th:exis-uniq}] {\bf
Existence} The proof of existence will be divided in two steps.

\noindent \textit{Step}\ 1.  $g $ does not dependent on $\left(
Y,Z\right)$.

\noindent Recall that $Y_{t}^{n}\nearrow Y_{t}$ $a.s$. Then,
Fatou's lemma and Lemma \ref{lem:1} ensure
\begin{eqnarray*}
\E\left( \sup_{0\leq t\leq T }e^{\mu A_t}\left| Y_{t}\right| ^{2}\right)
<+\infty,
\end{eqnarray*}
It then follows from Lemma \ref{lem:1} and Lebegue's dominated
convergence theorem  that
\begin{eqnarray}
\E\left( \int_{0}^{T }\left| Y_{s}^{n}-Y_{s}\right| ^{2} ds
\right)\longrightarrow 0,\,\,\ \mbox{as}\,\  n\rightarrow
\infty.\label{b9}
\end{eqnarray}

Next, we will prove that the sequence of processes\ $Z^n$\ converges
in $M^{2}(0,T;\R^{d})$ To this end, for\ $n \geq p \geq 1$,
It\^{o}'s formula provide
\begin{eqnarray*}
&&\left| Y_{t}^{n}-Y_{t}^{p}\right| ^{2}+\int_{0 }^{t }\|
Z_{s}^{n}-Z_{s}^{p}\| ^{2}ds  \nonumber \\
&=&2\int_{0 }^{t
}(Y_{s}^{n}-Y_{s}^{p})[f(s,Y_{s}^{n},Z_{s}^{n})-f(s,Y_{s}^{p},Z_{s}^{p})]ds+2\int_0^{t}
(Y_{s}^{n}-Y_{s}^{p})[\phi(s,Y_{s}^{n})-\phi(s,Y_{s}^{p})]dA_s\\
&&-2\int_{0}^{t}\langle Y_{s}^{n}-Y_{s}^{p}
,[Z_{s}^{n}-Z_{s}^{p}]\downarrow
dW_{s}\rangle+2\int_{0}^{t}(Y_{s}^{n}-Y_{s}^{p})(dK_{s}^{n}-dK_{s}^{p})
. \label{b10}
\end{eqnarray*}
From the same step as before, by using again assumptions
$({\bf H}_2)$, there exists a constant $C>0$, such that
\begin{eqnarray*}
&&\E\left\{\left| Y_{t}^{n}-Y_{t}^{p}\right|^{2}+
\int_{0}^{t}\left| Y_{s}^{n}-Y_{s}^{p}\right| ^{2}dA_s +
\int_{0}^{t}\left\|
Z_{s}^{n}-Z_{s}^{p}\right\| ^{2}ds\right\} \\
&\leq &
C\E\left\{\int_{0}^{t}|Y_{s}^{n}-Y_{s}^{p}|^{2}ds+\sup_{0\leq s\leq
T}\left(Y_{s}^{n}-S_{s}\right)^{-}K_{T}^{p} +\sup_{0\leq s\leq
T}\left( Y_{s}^{p}-S_{s}\right) ^{-} K_{T}^{n}\right\},
\end{eqnarray*}
which, by Gronwall lemma, H\"{o}lder inequality and Lemma \ref{lem:1} implies
\begin{eqnarray*}
&&\E\left\{\left| Y_{t}^{n}-Y_{t}^{p}\right|^{2} +\int_{0}^{t}\left| Y_{s}^{n}-Y_{s}^{p}\right| ^{2}dA_s+ \int_{0}^{t}\left\|
Z_{s}^{n}-Z_{s}^{p}\right\| ^{2}ds\right\}\nonumber\\ &\leq &
C\left\{\E\left(\sup_{0\leq s\leq T}|\left( Y_{s}^{n}-S_{s}\right)
^{-}|^{2}\right)\right\}^{1/2}+C\left\{ \E\left(\sup_{0\leq s\leq T}|\left(
Y_{s}^{p}-S_{s}\right) ^{-}|^{2}\right)\right\}^{1/2}. \label{b11'}
\end{eqnarray*}
Finally, from Burkh\"{o}lder-Davis-Gundy's inequality, we obtain
\begin{eqnarray*}
\E\left( \sup_{0\leq s\leq T}\left| Y_{s}^{n}-Y_{s}^{p}\right|
^{2}+\int_{0}^{t}\left| Y_{s}^{n}-Y_{s}^{p}\right| ^{2}dA_s+\int_{0}^{T}\left\| Z_{s}^{n}-Z_{s}^{p}\right\| ^{2}ds\right)
\longrightarrow 0,\,\, \mbox{ as }\,  n,p\longrightarrow
\infty,\label{b12}
\end{eqnarray*}
which provides that the sequence of processes $(Y^{n},Z^{n})$ is
Cauchy in the Banach space $S^{2}([0,T];\R)\times
M^{2}(0,T;\R^{d})$. Consequently, there exists a couple $(Y,Z)\in
S^{2}([0,T];\R)\times M^{2}(0,T;\R^{d})$ such that
\begin{eqnarray*}
\E\left\{\sup_{0\leq s\leq T}\left| Y_{s}^{n}-Y_{s}^{{}}\right|
^{2}+\int_{0}^{t}\left| Y_{s}^{n}-Y_{s}\right|
^{2}dA_s+\int_{0}^{T}\left\|Z^{n}_{s}-Z_{s}\right\|^{2}ds\right)
\rightarrow 0,\mbox{ as }n\rightarrow \infty .
\end{eqnarray*}
On the other hand, we rewrite (\ref{K:n}) as
\begin{equation}\label{K:n1}
K_t^n=Y_t^n-\xi-\int_0^t f(s,Y_s^n,Z_s^n) ds -\int_0^t \phi(s,Y_s^n)
dA_s-\int_0^t g(s) dB_s+\int_0^t Z^n_s \downarrow d W_s.
\end{equation}
By the convergence of $Y^n,\ Z^n$ (for a subsequence), the fact that
$f,\phi$ are continuous and
\begin{itemize}
\item $\sup_{n\geq0}|f(s,Y^n_s,Z_s)|\leq f_s+K\left\{(\sup_{n\geq0}|Y_s^n|)+\|Z_s\|\right\}$,
\item $\sup_{n\geq0}|\phi(s,Y^n_s)|\leq \phi_s+K\left\{(\sup_{n\geq0}|Y^n_s|)\right\}$,
\item $\mathbb{E}\int_0^T|f(s,Y^n_s,Z^n_s)-f(s,Y^n_s,Z_s)|^2ds\leq C\mathbb{E}\int_0^T\|Z^n_s-Z_s\|^2ds$
\end{itemize}
we get the existence of a  process $K$ which verifies for all
$t\in[0,T]$
$$\E\left| K_{t}^{n}-K_{t}^{{}}\right| ^{2}\longrightarrow 0$$
and such that $\mathbb{P}$-a.s. and for all $t\in[0,T]$,
$$Y_t=\xi+\int_0^tf(s,Y_s,Z_s)ds+\int_0^t\phi(s,Y_s)dA_s+K_t+\int_0^t g(s) dB_s-\int_0^tZ_s \downarrow dW_s.$$

It remains to show that $(Y,Z,K)$ solves the reflected BSDE $(\xi,
f,\phi,g,S)$. In this fact, since $(Y^{n}_t,K^{n}_t)_{0\leq t\leq
T}$ tends to $(Y_t,K_t)_{0\leq t\leq T}$ in
probability uniformly in t , the measure $dK^n$ converges to $dK$ weakly in
probability, so that
$\int_{0}^{T}(Y_{s}^{n}-S_{s})dK_{s}^{n}\rightarrow
\int_{0}^{T}(Y_{s}-S_{s})dK_{s}$ in probability as $n\rightarrow
\infty$. On the other hand, in view of Lemma \ref{lem:2}, $Y_{t}\geq
S_{t}$ a.s., and thus $\int_{0}^{T}(Y_{s}-S_{s})dK_{s}\geq 0$.
Moreover, $\int_{0}^{T}(Y_{s}^{n}-S_{s})dK_{s}^{n}=-n\int_{0}^{T}|
(Y_{s}^{n}-S_{s})^{-}| ^{2}ds \leq 0$ and passing to the limit we
get $\int_{0}^{T}(Y_{s}-S_{s})dK_{s}\leq 0$, which together with the
above proved $(ii)$ of the definition.

\noindent \textit{Step}\ 2. The general case. In light of the above
step, and for any $(\bar{Y},\bar{Z})\in S^{2}([0,T];\R)\times M^{2}(0,T;\R^{d})$, the BDSDE
\begin{equation*}
Y_{t}=\xi+\int_{0}^t f(s,{Y}_s,{Z}_s)ds
+\int_{0}^{t}\phi(s,Y_s)dA_s+\int_{0}^{t}g(s,\bar{Y}_s,\bar{Z}_s)\,dB_{s}
-\int_{0}^{t}Z_{s}\downarrow dW_{s}+K_t
\end{equation*}
has a unique solution $(Y,Z,K)\in S^{2}([0,T];\R)\times M^{2}(0,T;\R^{d})$. So, we can define the mapping
$$
\begin{array}{lrlll}
\Psi:&S^{2}([0,T];\R)\times M^{2}(0,T;\R^{d})&\longrightarrow&S^{2}([0,T];\R)\times M^{2}(0,T;\R^{d})\\
&(\bar{Y},\bar{Z})&\longmapsto&(Y,Z)=\Psi(\bar{Y},\bar{Z}).
\end{array}
$$
Now, let $(Y,Z),\ (Y',Z'),\ (\bar{Y},\bar{Z})$ and
$(\bar{Y'},\bar{Z'})\in S^{2}([0,T];\R)\times M^{2}(0,T;\R^{d})$ such that
$(Y,Z)=\Psi(\bar{Y},\bar{Z})$ and $(Y',Z')=\Psi(\bar{Y'},\bar{Z'})$.
Put $\Delta \eta=\eta-\eta'$ for $\eta=Y,\bar{Y},Z,\bar{Z}$. By
virtue of It\^o's formula, we have
\begin{eqnarray*}
&&\E e^{\mu t+\beta A_t}|\Delta Y_t|^2+\E\int_0^t e^{\mu s+\beta A_s}\|\Delta Z_s\|^2ds\\
&&=2\E\int_0^t e^{\mu s+\beta A_s}\Delta
Y_s\left\{f(s,{Y}_s,{Z}_s)-f(s,{Y'}_s,{Z'}_s)\right\}ds
+2\E\int_0^t e^{\mu s+\beta A_s} \Delta Y_s\left\{\phi(s,{Y}_s)-\phi(s,{Y'}_s)\right\}dA_s\\
 &&+2\E\int_0^t e^{\mu s+\beta A_s} \Delta Y_s d(\Delta K_s)+\int_0^te^{\mu s+\beta A_s}\left\|g(s,\bar{Y}_s,\bar{Z}_s)-g(s,\bar{Y'}_s,\bar{Z'}_s)\right\|^2ds\\
 &&-\mu\E\int_0^t
e^{\mu s+\beta A_s} \left|\Delta Y_s\right|^2 ds-\beta\E\int_0^t
e^{\mu s+\beta A_s} \left|\Delta Y_s\right|^2 dA_s.
\end{eqnarray*}
But since $\displaystyle\E\int_0^t e^{\mu s+\beta A_s} \Delta Y_s
d(K_s-K'_s)\leq 0$, then from $({\bf H}_2)$ there exists
$\alpha<\alpha'<1$ such that
\begin{eqnarray*}
&&\E e^{\mu t+\beta A_t}|\Delta Y_t|^2+\alpha\E\int_0^t e^{\mu s+\beta A_s}\|\Delta Z_s\|^2ds\\
&&\leq \left(\frac{1-\alpha'}{c}+1-\alpha'-\mu\right)\E\int_0^t
e^{\mu s+\beta A_s}|\Delta Y_s|^2 ds+\beta\E\int_0^t e^{\mu s+\beta
A_s}|\Delta
Y_s|^2 dA_s\\
&&+c\E\int_0^te^{\mu s+\beta A_s}|\Delta\bar{Y}_s|^{2}ds+\alpha\E\int_0^te^{\mu s+\beta A_s}|\Delta\bar{Z}_s|^2ds
\end{eqnarray*}
Next, denote $\displaystyle\gamma=\frac{1-\alpha'}{c}+1-\alpha'$ and
choosing \ $\mu$\ such that\ $\displaystyle
\mu-\gamma=\frac{\alpha'c}{\alpha}$, we obtain
\begin{eqnarray*}
&&\bar{c}\E\int_0^t
e^{\mu s+\beta A_s} \left|\Delta Y_s\right|^2 ds+|\beta|\E\int_0^t e^{\mu s+\beta A_s}|\Delta
Y_s|^2 dA_s+\E\int_0^t e^{\mu s+\beta A_s}\|\Delta Z_s\|^2ds\\
&&\leq \frac{\alpha}{\alpha'}\left(\bar{c}\E\int_0^t e^{\mu s+\beta
A_s}\left|\Delta \bar {Y}_s\right|^2ds +\E\int_0^te^{\mu
s+\beta A_s}\left\|\Delta\bar{Z}_s)\right\|^2ds\right),\\
&&\leq \frac{\alpha}{\alpha'}\left(\bar{c}\E\int_0^t e^{\mu s+\beta
A_s}\left|\Delta \bar {Y}_s\right|^2ds +|\beta|\E\int_0^t e^{\mu
s+\beta A_s}\left|\Delta \bar {Y}_s\right|^2dA_s+\E\int_0^te^{\mu
s+\beta A_s}\left\|\Delta\bar{Z}_s)\right\|^2ds\right)
\end{eqnarray*}
where  $\displaystyle \bar{c}=\frac{c}{\alpha}$.\\

Now, since $\displaystyle\frac{\alpha}{\alpha'}<1$, then it follows
that $\Psi$ is a strict contraction on
$\mathcal{S}^{2}([0,T],\R)\times \mathcal{M}^{2}((0,T);\R^d)$
equipped with the norm
\begin{eqnarray*}
\|(Y,Z)\|^{2}=\bar{c}\E\int_0^t
e^{\mu s+\beta A_s} \left|Y_s\right|^2 ds+|\beta|\E\int_0^t e^{\mu s+\beta A_s}|
Y_s|^2 dA_s+\E\int_0^t e^{\mu s+\beta A_s}\|Z_s\|^2ds
\end{eqnarray*}
and it has a unique fixed point, which is the unique solution  of
our BDSDE.

\noindent{\bf Uniqueness} Let us define
\begin{eqnarray*}
\left\{ \left( \Delta Y_{t},\Delta Z_{t},\Delta K_{t}\right) ,\mbox{
}0\leq t\leq T \right\} =\left\{ (Y_{t}-Y_{t}^{\prime
},Z_{t}-Z_{t}^{\prime },K_{t}-K_{t}^{\prime }),\mbox{ }0\leq t\leq T
\right\}
\end{eqnarray*}
where
 $\displaystyle{\left\{
\left( Y_{t},Z_{t},K_{t}\right) ,\mbox{ }0\leq t\leq T
\right\}}$ and $\displaystyle{\left\{ (Y_{t}^{\prime },Z_{t}^{\prime },K_{t}^{\prime }),%
\mbox{ }0\leq t\leq T\right\}}$ denote two solutions of the
reflected BDSDE associated to the data $(\xi,f,g,\phi,S)$.

It follows again by It\^{o}'s formula that for every $0\leq t\leq T$
\begin{eqnarray*}
&&\left| \Delta Y_{t}\right| ^{2}+\int_{0}^{t}\|\Delta Z_{s}\| ^{2}ds \\
&=& 2\int_{0}^{t}\Delta
Y_{s}(f(s,Y_{s},Z_{s}^{{}})-f(s,Y_{s}^{\prime },Z_{s}^{\prime
}))ds+\int_{0}^{t}\|g(s,Y_{s},Z_{s}^{{}})-g(s,Y_{s}^{\prime
},Z_{s}^{\prime })\|^{2}ds\nonumber\\
&&+2\int_{0}^{t}\Delta Y_{s}(\phi(s,Y_{s})-\phi(s,Y_{s}^{\prime
}))dA_s+\int_{0}^{t}\langle \Delta
Y_{s},(g(s,Y_{s},Z_{s}^{{}})-g(s,Y_{s}^{\prime },Z_{s}^{\prime
}))dB_{s}\rangle \nonumber\\
&&-2\int_{0}^{t}\langle \Delta Y_{s},\Delta Z_{s}dW_{s}\rangle
+2\int_{0}^{t}\Delta Y_{s}d(\Delta K_s).
\end{eqnarray*}
Since
\begin{eqnarray*}
\int^{T}_{0}\Delta Y_{s} d(\Delta K_s) \leq 0,
\end{eqnarray*}
and by using similar computation as in the proof of existence, we have
\begin{eqnarray*}
\E\left\{\left| \Delta Y_{t}\right| ^{2}+\int_{0}^{T}|\Delta
Y_{s}|dA_s+\int_{0}^{T}\| \Delta Z_{s}\| ^{2}ds \right\} &\leq
&C\E\int_{0}^{T}|\Delta Y_{s}|^2ds,
\end{eqnarray*}
from which, we deduce that $\displaystyle{\Delta Y_{t}=0}$ and
further $\displaystyle{\Delta Z_{t}=0}.$ On the other hand since
\begin{eqnarray*}
\Delta K_{t} &=&\Delta
Y_{t}-\int_{0}^{t}\left(f(s,Y_{s},Z_{s})-f(s,Y_{s}^{\prime},Z_{s}^{\prime
})\right)ds
-\int_{0}^{t}\left(\phi(s,Y_{s})-\phi(s,Y_{s}^{\prime})\right)dA_s\\
&&-\int_{0}^{t}\left(g(s,Y_{s},Z_{s})-g(s,Y_{s}^{\prime
},Z_{s}^{\prime })\right)dB_{s}+\int_{0}^{t}\Delta
Z_{s}\downarrow dW_{s},
\end{eqnarray*}
we have $\displaystyle{\Delta K_{t}=0}$. The proof is complete now.
\end{proof}
\section {Connection to stochastic viscosity solution for reflected SPDEs with nonlinear
Neumann boundary condition} \setcounter{theorem}{0}
\setcounter{equation}{0} In this section we will investigate the
reflected generalized BDSDEs studied in the previous
 section in order to give a probabilistic interpretation for the stochastic viscosity
 solution of a class of nonlinear reflected SPDEs with nonlinear Neumann boundary condition.
\subsection {Notion of stochastic viscosity solution for reflected SPDEs with nonlinear Neumann boundary condition}
With the same notations as in Section 2, let ${\bf
F}^{B}=\{\mathcal{F}_{t}^{B}\}_{0\leq t\leq T}$ be the filtartion
generated by $B$, where $B$ is a one dimensional Brownian motion. By
${\mathcal{M}}^{B}_{0,T}$ we denote all the ${\bf F}^{B}$-stopping
times $\tau$ such $0\leq \tau\leq T$, a.s.
${\mathcal{M}}^{B}_{\infty}$ is the set of all ${\bf
F}^{B}$-stopping times that are almost surely finite. For generic
Euclidean spaces $E$ and $E_{1}$ we introduce the following vector
spaces of functions:
\begin{enumerate}
\item The symbol $\mathcal{C}^{k,n}([0,T]\times
E; E_{1})$ stands for the space of all $E_{1}$-valued functions
defined on $[0,T]\times E$ which are $k$-times continuously
differentiable in $t$ and $n$-times continuously differentiable in
$x$, and $\mathcal{C}^{k,n}_{b}([0,T]\times E; E_{1})$ denotes the
subspace of $\mathcal{C}^{k,n}([0,T]\times E; E_{1})$ in which all
functions have uniformly bounded partial derivatives.
\item For any sub-$\sigma$-field $\mathcal{G} \subseteq
\mathcal{F}_{T}^{B}$, $\mathcal{C}^{k,n}(\mathcal{G},[0,T]\times E;
E_{1})$ (resp.\, $\mathcal{C}^{k,n}_{b}(\mathcal{G},[0,T]\times E;
E_{1})$) denotes the space of all $\mathcal{C}^{k,n}([0,T]\times E;
E_{1})$  (resp.\, $\mathcal{C}^{k,n}_{b}([0,T]\times E;E_{1})$-valued
random variable that are $\mathcal{G}\otimes\mathcal{B}([0,T]\times
E)$-measurable;
\item $\mathcal{C}^{k,n}({\bf F}^{B},[0,T]\times E; E_{1})$
(resp.$\mathcal{C}^{k,n}_{b}({\bf F}^{B},[0,T]\times E; E_{1})$) is
the space of all random fields $\phi\in
\mathcal{C}^{k,n}({\mathcal{F}}_{T},[0,T]\times E; E_{1}$ (resp.
$\mathcal{C}^{k,n}({\mathcal{F}}_{T},[0,T]\times E; E_{1})$, such
that for fixed $x\in E$, the mapping
$\displaystyle{(t,\omega_{1})\rightarrow \alpha(t,\omega_{1},x)}$ is
${\bf F}^{B}$-progressively measurable.
\item For any sub-$\sigma$-field $\mathcal{G} \subseteq
\mathcal{F}^{B}$ and a real number $ p\geq 0$,
$L^{p}(\mathcal{G};E)$ to be all $E$-valued $\mathcal{G}$-measurable
random variable $\xi$ such that $ \E|\xi|^{p}<\infty$.
\end{enumerate}
Furthermore, regardless their dimensions we denote by
$<.,.>$ and $|.|$ the inner product and norm in $E$ and
$E_1$, respectively. For $(t,x,y)\in[0,T]\times\R^{d}\times\R$, we
denote $D_{x}=(\frac{\partial}{\partial
x_{1}},....,\frac{\partial}{\partial x_{d}}),\,\\
D_{xx}=(\partial^{2}_{x_{i}x_{j}})_{i,j=1}^{d}$,
$D_{y}=\frac{\partial}{\partial y}, \,\
D_{t}=\frac{\partial}{\partial t}$. The meaning of $D_{xy}$ and
$D_{yy}$ is then self-explanatory.

Let $\Theta$ be an open connected bounded domain of $\R^{d}\, (d\geq
1)$. We suppose that $\Theta$ is smooth domain, which is such that
for a function $\psi\in\mathcal{C}^{2}_b(\R^{d}),\ \Theta$ and its
boundary  $\partial\Theta$ are characterized by
$\Theta=\{\psi>0\},\, \partial\Theta=\{\psi=0\}$ and, for any
$x\in\partial\Theta,\, \nabla\psi(x)$ is the unit normal vector
pointing towards the interior of $\Theta$.

In this section, we consider the continuous coefficients $f$ and
$\phi$,
\begin{eqnarray*}
f&:&\Omega_{1}\times[0,T]\times\overline{\Theta}\times\R\times\R^{d}\longrightarrow
\R\\
\phi&:&\Omega_{1}\times[0,T]\times\overline{\Theta}\times\R\longrightarrow
\R\\
\end{eqnarray*}
with the property that for all $x\in\overline{\Theta},\ f(.,x,.,.)$
and $\phi(.,x,.)$ are Lipschitz continuous in $x$ and satisfy the
conditions $({\bf H'}_1)$ and $({\bf H}_2)$, uniformly in $x$,
where, for some constant $K>0,$ the condition $({\bf H'}_1)$ is:
\begin{eqnarray*}
({\bf H'}_1)\left\{
\begin{array}{l}
|f(t,x,y,z)|\leq K(1+|x|+|y|+\|z\|),\\
|\phi(t,x,y)|\leq K(1+|x|+|y|).
\end{array}\right.
\end{eqnarray*}
Furthermore, we shall make use of the following assumptions:
\begin{description}
\item  $({\bf H}_3)$ The function $\sigma:\R^{d}\longrightarrow\R^{d\times
d}$ and $b:\R^{d}\longrightarrow\R^{d}$ are uniformly Lipschitz
continuous, with common Lipschitz constant $K>0$.
\item $({\bf H}_4)$ The functions $l:\overline{\Theta}\longrightarrow \R$ and $h:[0,T]\times\overline{\Theta}\longrightarrow \R$ are continuous such that, for some $K>0,$
\begin{eqnarray*}
|l(x)|&\leq& K(1+|x|)\\
|h(t,x)|&\leq& K(1+|x|)\\
h(0,x)&\leq& l(x),\,\,\,\,\,\,\,\,\,\,\,\,\,\,x\in\overline{\Theta}.
\end{eqnarray*}
\item $({\bf H}_5)$ The function
$g\in{\mathcal{C}}_{b}^{0,2,3}([0,T]\times\overline{\Theta}\times\R;\R)$.
\end{description}
Let us consider the related obstacle problem for SPDE with nonlinear
Neumann boundary condition:
\begin{eqnarray*}
\mathcal{OP}^{(f,\phi,g,h,l)}\left\{
\begin{array}{l}
\displaystyle \min\left\{u(t,x)-h(t,x),\; -\frac{\partial
u(t,x)}{\partial t}-[
Lu(t,x)+f(t,x,u(t,x),\sigma^{*}(x)D_{x}u(t,x))]dt\right.\\\\
\left.\,\,\,\,\,\,\,\,\,\,\,\,\,\,\,\,\,\ -
g(t,x,u(t,x))\lozenge B_{s}\right\}=0,\,\,\
(t,x)\in[0,T]\times\Theta\\\\ u(0,x)=l(x),\,\,\,\,\,\,\ x\in\overline{\Theta}\\\\
\displaystyle\frac{\partial u}{\partial
n}(t,x)+\phi(t,x,u(t,x))=0,\,\,\ (t,x)\in[0,T]\times\partial\Theta,
\end{array}\right.
\end{eqnarray*}
where
\begin{eqnarray*}
L=\frac{1}{2}\sum_{i,j=1}^{d}(\sigma(x)\sigma^{*}(x))_{i,j}
\frac{\partial^{2}}{\partial x_{i}\partial
x_{j}}+\sum^{d}_{i=1}b_{i}(x)\frac{\partial}{\partial x_{i}},\quad
\forall\, x\in\Theta,
\end{eqnarray*}
and
\begin{eqnarray*}
\frac{\partial}{\partial n}=\sum_{i=1}^{d}\frac{\partial
\psi}{\partial x_{i}}(x)\frac{\partial}{\partial x_{i}},\quad
\forall\, x\in\partial\Theta.
\end{eqnarray*}
As in the work of Buckdahn-Ma \cite{BM1,BM2}, our next goal is
to define the notion of stochastic viscosity to
$\mathcal{OP}^{(f,\phi,g,h)}$. So, we shall recall some of
their notation. Let $\eta\in\mathcal{C}({\bf
F}^{B},[0,T]\times\R^{d}\times\R)$ be the solution to the equation
\begin{eqnarray*}
\eta(t,x,y)&=&y+\int^{t}_{0}\langle g(s,x,\eta(s,x,y)),
\circ{dB}_s\rangle,
\end{eqnarray*}
where the stochastic integrals have to be interpreted in Stratonowich sense. We have the following relation with the standard It\^{o} integral:
\begin{eqnarray*}
\int^{t}_{0}\langle g(s,x,\eta(s,x,y)),
\circ{dB}_s\rangle&=&\frac{1}{2}\int^{t}_{0}\langle g,D_yg\rangle(s,x,\eta(s,x,y)ds
+\int^{t}_{0}\langle g(s,x,\eta(s,x,y)),dB_s\rangle.\label{c'1}
\end{eqnarray*}
Under the assumption $({\bf H}_5)$ the mapping $y\mapsto\eta(s,x,y)$ defines a diffeomorphism
for all $t,x, \, a.s$. Hence if we denote by $\varepsilon(s,x,y)$ its
$y$-inverse, one can show that (cf. Buckdahn and Ma \cite{BM1})
\begin{eqnarray}
\varepsilon(t,x,y)=y-\int^{t}_{0}\langle D_{y}\varepsilon(s,x,y)
g(s,x,y),\circ {dB}_{s}\rangle. \label{c2}
\end{eqnarray}
To simplify the notation in the sequel we denote
\begin{eqnarray*}
&&A_{f,g}(\varphi(t,x))=
L\varphi(t,x)+f(t,x,\varphi(t,x),\sigma^{*}D_{x}\varphi(t,x))-\frac{1}{2}
(g,D_{y}g) (t,x,\varphi(t,x))\\
&&\mbox{and}\, \Psi(t,x)=\eta(t,x,\varphi(t,x)).
\end{eqnarray*}
\begin{definition}\label{D:defvisco}
A random field $u \in \mathcal{C}\left(\mathbf{F}^B, [0,T]\times
\overline{\Theta}\right)$ is called a stochastic viscosity
subsolution of the stochastic obstacle problem
$\mathcal{OP}^{(f,\phi,g,h,l)}$ if $u\left(0,x\right)\leq
l\left(x\right)$, for all $x\in \overline{\Theta}$, and if for any
stopping time $\tau \in \mathcal{M}_{0,T}^B$, any state variable
$\xi\in L^0\left(\mathcal{F}_{\tau}^B, \Theta\right)$, and any
random field $\varphi\in
\mathcal{C}^{1,2}\left(\mathcal{F}_{\tau}^B,\ [0,T]\times
\mathbb{R}^d\right)$, with the property that for $\mathbb{P}$-almost
all $\omega\in\left\{0<\tau<T\right\}$ the inequality
$$u\left(t,\omega,x\right)-\Psi\left(t,\omega,x\right)
 \leq 0=u\left(\tau(\omega),\xi(\omega)\right)-\Psi
\left(\tau(\omega),\xi(\omega)\right)$$ is fulfilled for all
$\left(t,x\right)$ in some neighborhood
$\mathcal{V}\left(\omega,\tau\left(\omega\right),\xi\left(\omega\right)\right)$
of $\left(\tau\left(\omega\right),\xi\left(\omega\right)\right)$,
the following conditions are satisfied:
\begin{itemize}
\item[(a)] on the event $\left\{0<\tau<T\right\}\cap\left\{\xi\in
\Theta\right\}$ the inequality
\begin{equation}\label{E:def1}
\min\left\{u(\tau,\xi)-h(\tau,\xi),\mathrm{A}_{f,
g}\left(\Psi\left(\tau,\xi\right)\right)- D_y\Psi
\left(\tau,\xi\right)D_t \varphi\left(\tau,\xi\right)\right\}\leq 0
\end{equation}
holds, $\mathbb{P}$-almost surely;
\item[(b)] on the event $\left\{0<\tau<T\right\}\cap\left\{\xi\in
\partial \Theta\right\}$ the inequality
\begin{align} \min & \left[\min\left\{u(\tau,\xi)-h(\tau,\xi),\mathrm{A}_{f,
g}\left(\Psi\left(\tau,\xi\right)\right)- D_y \Psi
\left(\tau,\xi\right) D_t
\varphi\left(\tau,\xi\right)\right\}\,,\right.\nonumber\\
& \left.-\frac{\displaystyle{\partial \Psi}}{\displaystyle{\partial
n}}\left(\tau,\xi\right)-\phi\left(\tau,\xi,\Psi\left(\tau,\xi\right)\right)
\right] \leq 0 \label{E:viscosity01}
\end{align}
holds, $\mathbb{P}$-almost surely.
\end{itemize}
A random field $u \in \mathcal{C}\left(\mathbf{F}^B, [0,T]\times
\overline{\Theta}\right)$ is called a stochastic viscosity
supersolution of the stochastic obstacle problem
$\mathcal{OP}^{(f,\phi,g,h,l)}$ if $u\left(0,x\right)\geq
l\left(x\right)$, for all $x\in \overline{\Theta}$, and if for any
stopping time $\tau\in\mathcal{M}_{0,T}^B$, any state variable
$\xi\in L^0\left(\mathcal{F}_{\tau}^B,\Theta\right)$, and any random
field $\varphi\in \mathcal{C}^{1,2}\left(\mathcal{F}_{\tau}^B,\
[0,T]\times \mathbb{R}^d\right)$, with the property that for
$\mathbb{P}$-almost all $\omega\in\left\{0<\tau<T\right\}$ the
inequality
$$u\left(t,\omega,x\right)-\Psi \left(t,\omega,x\right)
 \geq 0=u\left(\tau(\omega),\xi(\omega)\right)-\Psi
\left(\tau(\omega),\xi(\omega)\right)$$ is fulfilled for all
$\left(t,x\right)$ in some neighborhood
$\mathcal{V}\left(\omega,\tau\left(\omega\right),\xi\left(\omega\right)\right)$
of $\left(\tau\left(\omega\right),\xi\left(\omega\right)\right)$,
the following conditions are satisfied:
\begin{itemize}
\item[(a)] on the event $\left\{0<\tau<T\right\}\cap\left\{\xi\in
\Theta\right\}$ the inequality
\begin{equation}\label{E:def12}
\min\left\{u(\tau,\xi)-h(\tau,\xi),\mathrm{A}_{f,
g}\left(\Psi\left(\tau,\xi\right)\right) - D_y\Psi
\left(\tau,\xi\right)D_t \varphi\left(\tau,\xi\right)\right\}\geq 0
\end{equation}
holds, $\mathbb{P}$-almost surely;
\item[(b)] on the event $\left\{0<\tau<T\right\}\cap\left\{\xi\in
\partial \Theta\right\}$ the inequality
\begin{align} \max & \left[\min\left\{u(\tau,\xi)-h(\tau,\xi),\mathrm{A}_{f,
g}\left(\Psi\left(\tau,\xi\right)\right)- D_y \Psi
\left(\tau,\xi\right) D_t
\varphi\left(\tau,\xi\right)\right\}\,,\right.\nonumber\\
& \left.-\frac{\displaystyle{\partial \Psi}}{\displaystyle{\partial
n}}\left(\tau,\xi\right)-\phi\left(\tau,\xi,\Psi\left(\tau,\xi\right)\right)
\right] \geq 0 \label{E:viscosity02}
\end{align}
 holds, $\mathbb{P}$-almost surely.
 \end{itemize}

Finally, a random field $u \in \mathcal{C}\left(\mathbf{F}^B,
[0,T]\times \overline{\Theta}\right)$ is called a stochastic
viscosity solution of the stochastic obstacle problem
$\mathcal{OP}^{(f,\phi,g,h,l)}$ if it is both a stochastic
viscosity subsolution and a supersolution.
\end{definition}
\begin{remark}
Observe that if $f,\,\phi$ are deterministic and $g\equiv 0$, the flow $\eta$ becomes $\eta(t,x,y)=y,\, \forall\ (t,x,y)$ and $\Psi(t,x)=\varphi(t,x)$. Thus,
definition\, $\ref{D:defvisco}$ coincides with the definition of
(deterministic) viscosity solution of PDE
$\mathcal{OP}^{(f,\phi,0,h,l)}$ given by Ren et al in \cite{Ral}.
\end{remark}
\subsection {Existence of stochastic viscosity solutions for SPDE
with nonlinear Neumann boundary condition} The main objective of
this subsection is to show how the stochastic obstacle problem
$\mathcal{OP}^{(f,\phi,g,h,l)}$ is related to reflected generalized
BDSDE $(\ref{a1})$ introduced in Section 1. For this end we recall
some known results on reflected diffusions. We consider
\allowdisplaybreaks\begin{align}
s\mapsto &A_{s}^{t,x} \,\,\,\hbox{is increasing}\nonumber\\
X_s^{t,x} &= x+\int_s^t b\left(X_r^{t,x}\right) dr+\int_s^t
\sigma\left(X_r^{t,x}\right) d\downarrow{W}_r+\int_s^t \nabla \psi
\left(X_r^{t,x}\right) dA_r^{t,x}, \quad
\forall\, s\in [0,t]\,,\nonumber\\
A^{t,x}_{s}&=\int_{s}^{t} I_{\left\{ X^{t,x}_{r}\in\partial\Theta
\right\}}\, dA^{t,x}_{r}.\label{rSDE}
\end{align}
We note here that due to the direction of the It\^{o} integral,
$(\ref{rSDE})$ should be viewed as going from $t$ to $0$ (i.e.,
$X^{t,x}_0$ should be understood as the terminal value of the
solution $X^{t,x}$ ). It is then clear (see \cite{LZ}) that under
conditions $({\bf H}_3)$ on the coefficients $b$ and $\sigma$,
$(\ref{rSDE})$ has a unique strong ${\bf F}^{W}$-adapted solution.
We refer to Pardoux and Zhang \cite{PZ}( Propositions 3.1 and 3.2),
and S{\l}omi\`nski \cite{Sl},for the following regularity results.
\begin{proposition}\label{P:continuity00}
There exists a constant $ C>0 $ such that for all for all $t\leq
t_1<t_2\leq T$ and $x_1,x_2\in \overline{\Theta}$,\  the following
inequalities hold:
$$\mathbb{E}\left[\sup_{t_2\leq s\leq
T}\left|X^{t_1,x_1}_{s}-X^{t_2,x_2}_{s}\right|^4\right] \leq
C\left\{ |t_2-t_1|^2+|x_1-x_2|^4.\right\}$$
and$$\mathbb{E}\left[\sup_{t_2\leq s\leq
T}\left|A_s^{t_1,x_1}-A_{s}^{t_2,x_2}\right|^4\right]\leq C\left\{
|t_2-t_1|^2+|x_1-x_2|^4.\right\}.$$ Moreover, for all $p\geq 1$,
there exists a constant $C_p$ such that for all
$(t,x)\in\mathbb{R}_+ \times \overline{\Theta}$,
\[\mathbb{E}\left(\left|A_s^{t,x}\right|^p \right) \leq C_p(1+t^p)\] and
for each $\mu$, $0<s<t$, there exists a constant $C(\mu,t)$ such
that for all $x\in\overline{\Theta}$,
$$\mathbb{E}\left(\displaystyle e^{\mu A_{s}^{t,x}}\right) \leq C(\mu,t).$$
\end{proposition}
Now, we consider the following reflected
generalized BDSDE: for $(t,x)\in[0,T]\times\overline{\Theta}$
\begin{align}\label{E:back}
\left\{
\begin{aligned}
&Y_s^{t,x}  =   l\left(X^{t,x}_0\right)+\int_{0}^s
f\left(r,X^{t,x}_r ,Y_r^{t,x},Z_r^{t,x}\right) dr+\int_{0}^s
g\left(r,X^{t,x}_r ,Y_r^{t,x}\right) dB_r
\\
& \qquad\qquad +\int_{0}^s \phi\left(r,X_r^{t,x}, Y_r^{t,x}\right)
dA_r^{t,x}+K_s^{t,x}-\int_{0}^s\left<Z_r^{t,x},\downarrow{dW_r}\right>,\\
& Y^{t,x}_s\geq h(s,X_s^{t,x})\, \mbox{such that}\ \displaystyle\int_0^T\left(Y^{t,x}_r-h(r,X_r^{t,x})\right)dK^{t,x}_r=0,\;\;\; 0\leq s\leq t.
\end{aligned}
\right.
\end{align}
where the coefficients $l$, $f$, $g$, $\phi$ and $h$ satisfy the
hypotheses $({\bf H'}_{1}),\, ({\bf H}_2), ({\bf H}_4)$ and $({\bf H}_5)$.
\begin{proposition}
Let the ordered triplet $(Y^{t,x}_s,Z^{t,x}_s,K^{t,x}_s) $ be a
solution of the BDSDE $(\ref{E:back})$. Then the random field
$(s, t, x)\mapsto Y^{t,x}_s , \, (s, t, x)\in[0,T ]\times[0,T
]\times \Theta$ is almost surely continuous.
\end{proposition}
\begin{proof}
If we denote by $\mathbb{E}^{\mathcal{F}_s}$ the conditional
expectation with respect to $\mathcal{F}_s$, then we can  show that
there exists a constant $ C>0 $ such that for all $(t,x)$,
$(t',x')\in [0,T]\times\overline{\Theta}$ the following inequality
holds\newpage
\begin{eqnarray*}
&&\left\vert
Y^{t,x}_{s}-Y^{t',x'}_{s}\right\vert^2\\
&& \leq C\mathbb{E}^{\mathcal{F}_s} \left[e^{\mu
k_T}\left|l(X^{t,x}_{0}) -l(X^{t',x'}_{0})\right|^2+ \int_0^T e^{\mu
k_r} \left|f\left(r,X^{t,x}_r,Y^{t,x}_r,Z^{t,x}_r\right)
-f\left(r,X^{t',x'}_r,Y^{t',x'}_r,Z^{t',x'}_r\right)\right|^2dr\right.\\
&&+\int_0^Te^{\mu
k_r}\left|\phi\left(r,X^{t,x}_r,Y^{t,x}_r\right)\right|^2
d|\overline{A}|_r+ \int_0^Te^{\mu
k_r}\left|\phi\left(r,X^{t,x}_r,Y^{t,x}_r\right)
-\phi\left(r,X^{t',x'}_r,Y^{t',x'}_r\right)\right|^2dA_r^{t',x'}\\
&&\left. +\int_0^T e^{\mu k_r}
\left(h(r,X^{t,x}_{r})-h(r,X^{t',x'}_{r})\right)d\Delta K_r\right],
\end{eqnarray*}
where  $\Delta K:=K^{t,x}-K^{t',x'}$, $\overline{A}=
A^{t,x}-A^{t',x'}$\ and\
$k\triangleq\left|\overline{A}\right|+A^{t',x'}$ where
$\left|\overline{A}\right|$ is the total variation of the process
$\overline{A}$. Using the assumptions $({\bf H'}_{1})$\, and $({\bf
H}_2)$, we get
\begin{eqnarray*}
&&\left\vert
Y^{t,x}_{s}-Y^{t',x'}_{s}\right\vert^2\\
&& \leq C\mathbb{E}^{\mathcal{F}_s} \left[e^{\mu
k_T}\left|l(X^{t,x}_{0}) -l(X^{t',x'}_{0})\right|^2+ \int_0^T e^{\mu
k_r} \left|X^{t,x}_{r}-X^{t',x'}_{r}\right|^2dr\right.\\
&&+ \int_0^Te^{\mu
k_r}\left|X^{t,x}_{r}-X^{t',x'}_{r}\right|^2dA_r^{t',x'}+\int_0^T
e^{\mu k_r} \left(h(r,X^{t,x}_{r})-h(r,X^{t',x'}_{r})\right)d\Delta
K_r\\
&&\left.+\sup_{0\leq s\leq T}e^{\mu
k_T}\left(1+\left|X^{t,x}_s\right|^2+\left|Y^{t,x}_s\right|^2\right)
\left|A^{t,x}-A^{t',x'}\right|_T \right].
\end{eqnarray*}
It follows using Proposition \ref{P:continuity00} that $\left|A^{t,x}-A^{t',x'}\right|_T\rightarrow 0\;\; \P$-a.s., and $\forall\, s\in[0,t],\, \left|X^{t,x}_s-X^{t',x'}_s\right|^{2}\rightarrow 0\;\; \P$-a.s. as $(t,x)\rightarrow (t',x')$. Thus, the continuity  follows from the continuity of the functions $l$ and
$h$.
\end{proof}

Let now define
\begin{eqnarray}
u(t,x)=Y^{t,x}_{t},\,\,\,\ (t,x)\in[0,T]\times\overline{\Theta}.
\label{c3}
\end{eqnarray}
\begin{theorem}
$u\in C({\bf F}^{B},[0,T]\times\overline{\Theta})$ is a stochastic viscosity
solution of obstacle problem $\mathcal{OP}^{(f,\phi,g,h,l)}$.
\end{theorem}
\begin{proof}
For each $(t,x)\in[0,T]\times\overline{\Theta}, n\geq 1$, let
$\{^{n}Y^{t,x}_{s},{}^{n}Z^{t,x}_{s},\,\ 0\leq s\leq t\}$ denote the
solution of the generalized BDSDE
\begin{eqnarray*}
^{n}Y^{t,x}_{s}&=&l(X_{0}^{t,x})+\int^{s}_{0}f(r,X_{r}^{t,x},{}^{n}Y^{t,x}_{r},
{}^{n}Z^{t,x}_{r})dr+n\int^{s}_{0}({}^{n}Y^{t,x}_{r}-h(r,X^{t,x}_{r}))^{-}dr\nonumber\\
&&+\int^{s}_{0}\phi(r,X_{r}^{t,x},^{n}Y^{t,x}_{r})dA_r^{t,x}
\int^{s}_{0}g(r,X_{r}^{t,x},^{n}Y^{t,x}_{r})dB_{r}-\int^{s}_{0}
{}^{n}Z^{t,x}_{r}\downarrow dW_{r}. \label{c4}
\end{eqnarray*}
It is know from Boufoussi et al \cite{Bal} that
\begin{eqnarray*}
u_{n}(t,x)={}^{n}Y^{t,x}_{t},\,\,\,\ (t,x)\in
[0,T]\times\overline{\Theta},
\end{eqnarray*}
is the stochastic viscosity solution of the parabolic SPDE:
\begin{align}
\left\{
\begin{array}{l}
\frac{\displaystyle\partial u_{n}(t,x) }{\displaystyle\partial t}+[Lu_{n}(t,x)+f_{n}(t,x,u_{n}(t,x),\sigma^{*}D_{x}u_{n}(t,x))]+g(t,x,u_{n}(t,x))\lozenge B_{t}=0,\\\\
(t,x)\in [0,T]\times\Theta,\\\\
u_{n}(0,x)=l(x),\,\,\,\ x\in \overline{\Theta},\\\\
\frac{\displaystyle \partial u_{n}}{\displaystyle \partial
n}(t,x)+\phi(t,x,u_{n}(t,x))=0, \quad (t,x) \in [0,T]\times\partial
\Theta.
\end{array}\right.\label{SPDE}
\end{align}
where $f_{n}(t,x,y,z)=f(t,x,y,z)+n(y-h(t,x))^{-}$.

However, from the results of the previous section, for each
$(t,x)\in[0,T]\times\R^{d}$,
\begin{eqnarray*}
u_{n}(t,x)\uparrow u(t,x) \,\,\ a.s.\,\,\,\,\,\ \mbox{as}\,\
n\rightarrow \infty.
\end{eqnarray*}

Since $u_{n}$ and $u$ are continuous, it follows from Dini's theorem
that the above convergence is uniform on any compacts.

We now show
that $u$ is a stochastic viscosity subsolution of obstacle problem
of $\mathcal{OP}^{(f,\phi,g,h,l)}$. Let
$(\tau,\xi)\in\mathcal{M}^{B}_{0,T}\times L^{0}({\mathcal{F}}^{B}_{\tau};\overline{\Theta})$ satisfying $u(\tau,\xi)>h(\tau,\xi), \ \P$-a.s. and $\varphi\in
\mathcal{C}^{1,2}\left(\mathcal{F}_{\tau}^B,\ [0,T]\times
\overline{\Theta}\right)$ such that for $\P$-almost all $\omega\in\{0<\tau<T\}$, we have
\begin{eqnarray}
u(\omega,t,x)-\Psi(\omega,t,x)< 0=u(\omega,\tau(\omega),\xi(\omega))-\Psi(\omega,\tau(\omega),\xi(\omega))\label{V}
\end{eqnarray}
for all $(t,x)$ in some neighborhood
$\mathcal{V}\left(\omega,\tau\left(\omega\right),\xi\left(\omega\right)\right)$
of $\left(\tau\left(\omega\right),\xi\left(\omega\right)\right)$.

According the classical Lemma 6.1 in \cite{CL}, there exists  sequence of random variables $(\tau_{k},\xi_{k})_{k\geq 0}$ such that $(\tau_{k},\xi_{k})\rightarrow(\tau,\xi)$, $\P$-a.s., and $\varphi_k\in\mathcal{C}^{1,2}\left(\mathcal{F}_{\tau_k}^B,\ [0,T]\times\overline{\Theta}\right)$ satisfying $\varphi_k\rightarrow \varphi,\ \P$-a.s.,  such that
\begin{eqnarray*}
u_{n_k}(\omega,t,x)-\Psi_k(\omega,t,x)<
0=u_{n_k}(\omega,\tau_{k}(\omega),\xi_{k}(\omega))-\Psi_k(\omega,\tau_{k}(\omega),\xi_{k}(\omega))\label{V1}
\end{eqnarray*}
for all $(t,x)$ in some neighborhood
$\mathcal{V}\left(\omega,\tau_{k}\left(\omega\right),\xi_{k}\left(\omega\right)\right)\subset
\mathcal{V}\left(\omega,\tau\left(\omega\right),\xi\left(\omega\right)\right)$
for $k$ large enough.

On other hand, for $k$ enough large, let us define
\begin{eqnarray*}
\bar{\tau}_k=\inf \{t, \; u_{n_k}(t,x)-\Psi_k(t,x)=0\},\, \; x\in \overline{\Theta}.
\end{eqnarray*}
It easily seen that $(\tilde{\tau}_k)_k=(\bar{\tau}_k)_k\cap
(\tau_{k})_k$ is a sequence of stopping time satisfied
$\tilde{\tau}_k \rightarrow \tau$. Moreover, denoting by
$(\tilde{\xi}_{k})_{k}$ the subsequence of $(\xi_{k})_{k}$
associated to $(\tilde{\tau}_{k})_{k}$, it follows that
$(\tilde{\tau}_{k},\tilde{\xi}_{k})\in{\mathcal{M}}^{B}_{0,T}\times
L^{0}({\mathcal{F}}^{B}_{\tau_k};\overline{\Theta})$ and
\begin{eqnarray}
u_{n_k}(\omega,t,x)-\Psi_k(\omega,t,x)<
0=u_{n_k}(\omega,\tilde{\tau}_{k}(\omega),\tilde{\xi}_{k}(\omega))-\Psi_k(\omega,\tilde{\tau}_{k}(\omega),\tilde{\xi}_{k}(\omega))\label{V1}
\end{eqnarray}
for all $(t,x)$ in some neighborhood
$\mathcal{V}(\omega,\tilde{\tau}_{k}\left(\omega\right),\tilde{\xi}_{k}\left(\omega\right))\subset
\mathcal{V}\left(\omega,\tau\left(\omega\right),\xi\left(\omega\right)\right)$
for $k$ large enough.

Thus, since $u_n$ is a viscosity solution of SPDE $(\ref{SPDE})$ and according to $(\ref{V1})$, we
get:
\begin{itemize}
\item[(a)] On the event $\{0<\tilde{\tau}_{k}<T\}\cap\{\tilde{\xi}_{k}\in\Theta\}$ the inequality
\begin{equation*}
\mathrm{A}_{f_{n_k},g}\left(\Psi_k\left(\tilde{\tau}_{k},\tilde{\xi}_{k}\right)\right)-
D_y\Psi_k \left(\tilde{\tau}_{k},\tilde{\xi}_{k}\right)D_t
\varphi_k\left(\tilde{\tau}_{k},\tilde{\xi}_{k}\right)\leq 0
\end{equation*}
holds, $\mathbb{P}$-a.s.
\item[(b)] On the event$\{0<\tilde{\tau}_{k}<T\}\cap\{\tilde{\xi}_{k}\in\partial \Theta\}$ the inequality
\begin{eqnarray*}
\min\left[
\mathrm{A}_{f_{n_k},g}\left(\Psi_k\left(\tilde{\tau}_{k},\tilde{\xi}_{k}\right)\right)-
D_y\Psi_k \left(\tilde{\tau}_{k},\tilde{\xi}_{k}\right)D_t
\varphi_k\left(\tilde{\tau}_{k},\bar{\xi}_{k}\right),\right.\\
\left.-\frac{\partial\Psi_k}{\partial
n}(\tilde{\tau}_{k},\tilde{\xi}_{k})-\phi(\tilde{\tau}_{k},\tilde{\xi}_{k},\Psi_k(\tilde{\tau}_{k},\tilde{\xi}_{k}))\right]\leq
0
\end{eqnarray*}
holds, $\mathbb{P}$-a.s.
\end{itemize}

From the assumption that $u(\tau,\xi)>h(\tau,\xi)$ and the uniform convergence of $u_n$, it follows that for $k$ large enough $u_{n_k}(\tilde{\tau}_k,\tilde{\xi}_k)>h(\tilde{\tau}_k,\tilde{\xi}_k)$.

Therefore, taking the
limit as $k\rightarrow\infty$ in the above inequality yields:
\begin{itemize}
\item[(a)] On the event $\{0<\tau<T\}\cap\{\xi\in\Theta\}$ the inequality
\begin{equation*}
\mathrm{A}_{f,g}\left(\Psi\left(\tau,\xi\right)\right)- D_y\Psi
\left(\tau,\xi\right)D_t \varphi\left(\tau,\xi\right)\leq 0
\end{equation*}
holds, $\mathbb{P}$-a.s.
\item[(b)] On the event$\{0<\tau<T\}\cap\{\xi\in\partial \Theta\}$ the inequality
\begin{eqnarray*}
\min\left[ \mathrm{A}_{f,g}\left(\Psi\left(\tau,\xi\right)\right)-
D_y\Psi \left(\tau,\xi\right)D_t
\varphi\left(\tau,\xi\right),\right.\\
\left.-\frac{\partial\Psi}{\partial
n}(\tau,\xi)-\phi(\tau,\xi,\Psi(\tau,\xi))\right]\leq 0
\end{eqnarray*}
holds, $\mathbb{P}$-a.s.
\end{itemize}
This proved that $u$ is a stochastic viscosity subsolution of
$\mathcal{OP}^{f,\phi,g,h,l}$.

By the same argument as above one can
show that $u$ given by $(\ref{c3})$ is also a stochastic viscosity
supersolution of $\mathcal{OP}^{f,\phi,g,h,l}$.

We conclude that $u$ is a stochastic viscosity of $\mathcal{OP}^{f,\phi,g,h,l}$, which end the proof.
\end{proof}

\noindent{\bf Acknowledgments}\newline The first author would like
to express his deep gratitude to B. Boufoussi, Y. Ouknine and UCAM
Mathematics Department for their friendly hospitality. An anonymous
referee is acknowledged for his comments, remarks and for a
significant improvement to the overall presentation of this paper.

\label{lastpage-01}

\begin{thebibliography}{99}
\bibitem{Bal} Boufoussi, B.; van Casteren, J.; Mrhardy, N. Generalized Backward
doubly stochastic differential equations and SPDEs with nonlinear
Neumann boundary conditions. {\sl Bernoulli} {\bf 13} (2007), no.2,
$423-446$.

\bibitem{BM1} Buckdahn R.; Ma J. Stochastic viscosity solutions for nonlinear stochastic partial differential equations. Part I, {\sl Stochastic process. Appl.} {\bf 93} (2001), no.2, $181-204$.

\bibitem{BM2}Buckdahn R.; Ma J. Stochastic viscosity solutions for nonlinear stochastic partial differential equations. Part II, {\sl Stochastic process. Appl.} {\bf 93} (2001), no.2, $205-228$.

\bibitem{BM3} Buckdahn R.;  Ma J. Pathwise stochastic Taylor expansions and stochastic viscosity solutions for fully nonlinear stochastic PDEs. {\sl Ann. Appl. Probab.} {\bf 30} (2002), no.3, $1131-1171$.

\bibitem{CL} Crandall M.; Ishii H.; Lions P.L. User's guide to the viscosity solutions of second order partial differential equations. {\sl Bull. A.M.S.} {\bf 27} (1992), no.1, $1-67$.

\bibitem{DM} Dellacherie, C.; Meyer, P. Probabilities and Potential. North-Holland Mathematics Studies, {\bf 29}. North-Holland Publishing Co., Amsterdam-New York; North-Holland Publishing Co., Amsterdam-New York, 1978. viii+189 pp.  North Holland, $(1978)$

\bibitem{Kal} El Karoui, N.; Kapoudjian, C.; Pardoux, E.; Peng, S.; Quenz, M. C. Reflected solution of backward SDE's, and related obstacle problem for PDE's, {\sl Ann. Probab.} {\bf 25} (1997), no.2, $702-737$.

\bibitem{LS} Lions P-L.; Souganidis P.E. Fully nonlinear stochastic partial differential equations. {\sl C.R. Acad. Sc. Paris, Sér. I Math.} {\bf 326} (1998), no.1, $1085-1092$.

\bibitem{LZ} Lions, P.L.; Sznitman, A.S. Stochastic differential equation with reflecting boundary conditions. {\sl Comm. Pure Appl. Math.}
{\bf 37}, (1984), no. 4, $511-537$

\bibitem{PP1}  Pardoux E.; Peng S. Adapted solution of backward stochastic differential equation,
\text{ \sl Systems Control Lett.} {\bf 4} (1990), no.1, $55-61$.

\bibitem{PP3} Pardoux, E.; Peng, S. Backward doubly stochastic differential equations and systems of quasilinear SPDEs. {\sl Probab. Theory Related Fields} {\bf 98}, (1994), no.2, $209-227$.

\bibitem{PZ}  Pardoux, E.; Zhang, S. Generalized BSDEs and nonlinear Neumann
boundary value problems, {\sl Probab. Theory Related Fields } {\bf
110} (1998), no.4, $535-558$.



\bibitem{Ral} Ren, Y.; Xia, N. Generalized reflected BSDE and obstacle problem for PDE with nonlinear Neumann boundary condition. {\sl Stoch. Anal. Appl.}
{\bf 24} (2006), no.5, $1013-1033,\ $.

\bibitem{Sl} S{\l}omi\`nski, L.
Euler's approximations of solutions of SDEs with reflecting
boundary. {\sl Stochastic process. Appl.} {\bf 94}, (2001), $317-337$.

\bibitem{Yal} Yufen S.; Yanling G.; Kai L. Comparison theorem of backward doubly stochastic differential equations and application. {\sl Stoch. Anal. Appl. } {\bf 23} (2005), no.1, $97-110$.
\end{thebibliography}
\end{document}